\documentclass[11pt]{article}
\usepackage[english]{babel}
\usepackage[a4paper,scale=0.75,hcentering]{geometry}
\usepackage[hyperfootnotes=false]{hyperref} 
\usepackage{setspace}
\hypersetup{
	colorlinks = true,
	linkcolor = {blue},
	urlcolor = {red},
	citecolor = {blue}
}
\usepackage{amsmath, amsthm, amssymb, mathrsfs, mathtools, enumitem}
\usepackage[x11names]{xcolor}

\allowdisplaybreaks

\newcounter{results}[section] 

\theoremstyle{plain}
\newtheorem{theorem}[results]{Theorem}
\newtheorem{lemma}[results]{Lemma}

\newtheorem{corollary}[results]{Corollary}

\newtheorem*{theorem*}{Theorem}
\newtheorem*{lemma*}{Lemma}
\newtheorem*{proposition*}{Proposition}
\newtheorem*{corollary*}{Corollary}
\newtheorem*{exercise*}{Exercise}
\newtheorem*{fact*}{Fact}
\newtheorem*{problem*}{Problem}
\newtheorem*{conjecture*}{Conjecture}

\theoremstyle{remark}
\newtheorem{remark}[results]{Remark}

\newtheorem*{remark*}{Remark}
\newtheorem*{question*}{Question}

\theoremstyle{definition}

\newtheorem*{definition*}{Definition}
\newtheorem*{example*}{Example}

\numberwithin{equation}{section}

\newcommand{\N}{\ensuremath{\mathbb N}} 
\newcommand{\R}{\ensuremath{\mathbb R}} 
\providecommand{\C}{}
\renewcommand{\C}{\ensuremath{\mathcal C}} 
\renewcommand{\S}{\ensuremath{\mathbb S}} 

\newcommand{\I}{\ensuremath{\mathcal I}}

\renewcommand{\O}{\ensuremath{\mathcal O}}
\newcommand{\M}{\ensuremath{\mathcal M}}
\makeatletter 
\DeclarePairedDelimiter{\@tmpabs}{\lvert}{\rvert}
\newcommand{\@absstar}[1]{{\@tmpabs*{#1}}}
\newcommand{\@absnostar}[2][]{{\@tmpabs[#1]{#2}}}
\newcommand{\abs}{\@ifstar\@absstar\@absnostar}
\makeatother

\makeatletter 
\DeclarePairedDelimiter{\@tmpnorm}{\lVert}{\rVert}
\newcommand{\@normstar}[1]{{\@tmpnorm*{#1}}}
\newcommand{\@normnostar}[2][]{{\@tmpnorm[#1]{#2}}}
\newcommand{\norm}{\@ifstar\@normstar\@normnostar}
\makeatother

\let\div\undefined
\DeclareMathOperator{\div}{div}

\ifdefined\comma
	\renewcommand{\comma}{\ensuremath{\, \text{, }}}
\else
	\newcommand{\comma}{\ensuremath{\, \text{, }}}
\fi

\definecolor{removed}{RGB}{255, 0, 0}
\definecolor{added}{RGB}{0,192,0}

\usepackage[hyperfootnotes=false]{hyperref} 
\usepackage{setspace}
\hypersetup{
	colorlinks = true,
	linkcolor = {blue},
	urlcolor = {red},
	citecolor = {blue}
}

\usepackage[nameinlink,capitalise,sort]{cleveref} 
\crefname{equation}{}{} 
\crefname{enumi}{}{} 

\allowdisplaybreaks

\numberwithin{equation}{section}
\begin{document}

\title{\bf Degenerate stability of critical points of the Caffarelli-Kohn-Nirenberg inequality along the Felli-Schneider curve\thanks{This work is supported by  National Key R$\&$D Program of China (Grant 2023YFA1010001) and NSFC(12571123). E-mail addresses: zhou-yx22@mails.tsinghua.edu.cn(Zhou),    zou-wm@mail.tsinghua.edu.cn (Zou)}}

\author{{\bf Yuxuan Zhou, Wenming Zou}\\ {\footnotesize \it  Department of Mathematical Sciences, Tsinghua University, Beijing 100084, China.} }

\date{}



\maketitle

\begin{abstract}
{\small In this paper, we investigate the validity of a quantitative version of stability for the critical Hardy-H\'enon equation
\begin{equation*}
    H(u):=\div(|x|^{-2a}\nabla u)+|x|^{-pb}|u|^{p-2}u=0,\quad u\in D_a^{1,2}(\R^n),
\end{equation*}
\begin{equation*}
    n\geq 2,\quad a<b<a+1,\quad a<\frac{n-2}{2},\quad p=\frac{2n}{n-2+2(b-a)},
\end{equation*}
which is well known as the Euler-Lagrange equation of the classical Caffarelli-Kohn-Nirenberg inequality. Establishing quantitative stability for this equation amounts to finding a nonnegative function $F$ such that the estimate
\begin{equation*}
    \inf_{\substack{U_i\in\mathcal{M}\\
    1\leq i\leq\nu}}\norm*{u-\sum_{i=1}^\nu U_i}_{D_a^{1,2}(\R^n)}\leq C(a,b,n)F(\norm*{H(u)}_{D_a^{-1,2}(\R^n)})
\end{equation*}
holds for any nonnegative function $u$ satisfying
\begin{equation*}
    \left(\nu-\frac{1}{2}\right)S(a,b,n)^{\frac{p}{p-2}}\leq\int_{\R^n}|x|^{-2a}|\nabla u|^2\mathrm{d}x\leq \left(\nu+\frac{1}{2}\right)S(a,b,n)^{\frac{p}{p-2}}.
\end{equation*}
Here $\nu\in\N_+$ and $\mathcal{M}$ denotes the set of positive solutions of the equation. When $(a,b)$ falls above the Felli-Schneider curve, Wei and Wu \cite{Wei} found an optimal $F$. Their proof relies heavily on the fact that $\mathcal{M}$ is non-degenerate. When $(a,b)$ falls on the Felli-Schneider curve, due to the absence of the non-degeneracy condition, it becomes complicated and technical to find a suitable $F$. In this paper, we focus on this case. When $\nu=1$, we obtain an optimal $F$. When $\nu\geq2$ and $u$ is not too degenerate, we also derive an optimal $F$. To our knowledge, the results in this paper provide the first instance of degenerate stability in the critical point setting. We believe that our methods will be useful in other works on degenerate stability.
    \vskip0.1in
\noindent{\bf Key words:}  Caffarelli-Kohn-Nirenberg inequality, quantitative stability, critical point, Felli-Schneider curve.

\vskip0.1in
\noindent{\bf 2020 Mathematics Subject Classification:} Primary 46E35; 35J61; 26D10

}
\end{abstract}

\section{Introduction}
In this paper, we are concerned with the quantitative stability for the critical Hardy-H\'enon equation
\begin{equation}\label{eq}
    H(u):=\div(|x|^{-2a}\nabla u)+|x|^{-pb}|u|^{p-2}u=0,
\end{equation}
where $u\in D_a^{1,2}(\R^n)$ and $n,a,b,p$ satisfy the relations
\begin{equation}\label{cond}
    n\geq 2,\quad a<b<a+1,\quad a<\frac{n-2}{2},\quad p=\frac{2n}{n-2+2(b-a)}.
\end{equation}
The space $D_a^{1,2}(\R^n)$ is the completion of $C_c^{\infty}(\R^n)$ with respect to the norm
\begin{equation*}
    \norm*{u}_{D_a^{1,2}(\R^n)}:=\left(\int_{\R^n}|x|^{-2a}|\nabla u|^2\mathrm{d}x\right)^{\frac{1}{2}}.
\end{equation*}
It is well known that the equation \eqref{eq} is the Euler-Lagrange equation of the classical Caffarelli-Kohn-Nirenberg inequality (see \cite{Caf})
\begin{equation}\label{ckn}
    S(a,b,n)\left(\int_{\R^n}|x|^{-bp}|u|^p\mathrm{d}x\right)^{\frac{2}{p}}\leq\int_{\R^n}|x|^{-2a}|\nabla u|^2\mathrm{d}x.
\end{equation}
Here $S(a,b,n)$ denotes the sharp constant.

The classification of positive solutions to the equation \eqref{eq} has been extensively studied. Here, we provide a brief summary; for further background and partial results, we refer the readers to \cite{Cat,Dol,Dol1,Dol2,Lin} and their references. When $a\geq 0$, Chou and Chu \cite{Cho} found that the set of positive solutions is given by
\begin{equation}\label{sol}
    \mathcal{M}_{a,b,n}:=\left\{U_\lambda(x)\;\Big|\;U_\lambda(x):=\frac{\lambda^{\sqrt{\Lambda}}(2p\Lambda)^{\frac{1}{p-2}}}{(1+|\lambda x|^{\sqrt{\Lambda}(p-2)})^{\frac{2}{p-2}}},\lambda>0\right\},
\end{equation}
where $\Lambda:=\left(\frac{n-2-2a}{2}\right)^2$. Later, Catrina and Wang \cite{Cat} showed that for general $a,b,n$, the set $\mathcal{M}_{n,a,b}$ in \eqref{sol} gives all positive radially symmetric solutions. They also discovered that there exists a curve $b=f(a),a<0$ such that if $b<f(a)$, then there exists a positive non-radially symmetric solutions to the equation \eqref{eq}. Making use of fine spectral analysis, Felli and Schneider \cite{Fel} then proved that this curve can be taken to be
\begin{equation}\label{FS}
    b=b_{\text{FS}}(a):=\frac{n(n-2-2a)}{2\sqrt{(n-2-2a)^2+4n-4}}-\frac{n-2-2a}{2},\;a<0.
\end{equation}
The curve $b=b_{\text{FS}}(a)$ is usually called the Felli-Schneider curve. It was conjectured for a long time that when $a<0$ and $b\geq b_{\text{FS}}(a)$, $\mathcal{M}_{a,b,n}$ contains all positive solutions to the equation \eqref{eq}. This was finally confirmed by Dolbeault, Esteban and Loss \cite{Dol} using the method of nonlinear flows. In conclusion, when $a\geq 0$ or $a<0,b\geq b_{\text{FS}}(a)$, the set $\mathcal{M}_{n,a,b}$ in \eqref{sol} gives all the positive solutions to the equation \eqref{eq}. In the following, we will call the functions in $\mathcal{M}_{a,b,n}$ \textbf{Talenti bubbles}.

\vskip0.1in
The stability of the equation \eqref{eq}, or in other words, the stability of the Caffarelli-Kohn-Nirenberg inequality \eqref{ckn} in the critical point setting, can be understood as a natural extension of the classification results. Specifically, it refers to whether being almost a positive solution to the equation \eqref{eq} implies being close to an appropriate linear combination of Talenti bubbles. This kind of stability traces back to the celebrated global compactness principle of Struwe \cite{Str}. In \cite{Str}, Struwe established the following qualitative stability of the classical Yamabe equation in $\R^n\;(n\geq 3)$: for any $\varepsilon>0$ and $\nu\in\N_+$, there exists $\delta=\delta(n,\nu,\varepsilon)>0$ such that
\begin{equation*}
    \inf_{\substack{W_i\in\mathcal{M}_0\\
    1\leq i\leq\nu}}\norm*{u-\sum_{i=1}^\nu W_i}_{D_0^{1,2}(\R^n)}\leq \varepsilon\quad\text{if }\,\norm*{\Delta u+u^{2^*-1}}_{D_0^{-1,2}(\R^n)}\leq\delta.
\end{equation*}
Here $\mathcal{M}_0$ denotes the set of positive solutions to the Yamabe equation and $u$ is any nonnegative function satisfying
\begin{equation*}
    \left(\nu-\frac{1}{2}\right)S(0,0,n)^{\frac{n}{2}}\leq \int_{\R^n}|\nabla u|^2\mathrm{d}x\leq \left(\nu+\frac{1}{2}\right)S(0,0,n)^{\frac{n}{2}}.
\end{equation*}
Ciraolo, Figalli and Maggi \cite{Cir} were the first to quantify Struwe's work. In the case $\nu=1$, they proved the following sharp linear estimate
\begin{equation*}
    \inf_{W\in\mathcal{M}_0}\norm*{u-W}_{D_0^{1,2}(\R^n)}\leq C(n)\norm*{\Delta u+u^{2^*-1}}_{D_0^{-1,2}(\R^n)}.
\end{equation*}
This estimate was later generalized by Figalli and Glaudo \cite{Fig}. Assuming $\nu\geq 2$ and $3\leq n\leq 5$, they proved
\begin{equation*}
    \inf_{\substack{W_i\in\mathcal{M}_0\\
    1\leq i\leq\nu}}\norm*{u-\sum_{i=1}^\nu W_i}_{D_0^{1,2}(\R^n)}\leq C(n,\nu)\norm*{\Delta u+u^{2^*-1}}_{D_0^{-1,2}(\R^n)}.
\end{equation*}
Moreover, they constructed counterexamples to show that such a linear estimate does not hold for $n\geq 6$. The remaining case $\nu\geq2,n\geq 6$ was eventually solved by Deng, Sun and Wei \cite{Den}. They obtained the following sharp estimate
\begin{equation*}
    \inf_{\substack{W_i\in\mathcal{M}_0\\
    1\leq i\leq\nu}}\norm*{u-\sum_{i=1}^\nu W_i}_{D_0^{1,2}(\R^n)}\leq C(n,\nu)F_0\left(\norm*{\Delta u+u^{2^*-1}}_{D_0^{-1,2}(\R^n)}\right),
\end{equation*}
where
\begin{equation*}
    F_0(x):=\begin{cases}
        x|\ln{x}|^{\frac{1}{2}}+x,&n=6;\\
        x^{\frac{n+2}{2(n-2)}},&n\geq7.
    \end{cases}
\end{equation*}
Recently, Wei and Wu \cite{Wei} generalized the results in \cite{Den} to the critical Hardy-H\'enon equation \eqref{eq}. In the case $a\geq 0$ or $a<0,b>b_{\text{FS}}(a)$, they derived the following sharp estimate
\begin{equation}\label{www}
     \inf_{\substack{\lambda_i\in\R_+\\
    1\leq i\leq\nu}}\norm*{u-\sum_{i=1}^\nu U_{\lambda_i}}_{D_a^{1,2}(\R^n)}\leq C(n,a,b,\nu)F_1\left(\norm*{H(u)}_{D_a^{-1,2}(\R^n)}\right),
\end{equation}
where
\begin{equation*}
    F_1(x):=\begin{cases}
        x,&p>3\;\text{or }\nu=1;\\
        x|\ln{x}|^{\frac{1}{2}}+x,&p=3\;\text{and }\nu\geq 2;\\
        x^{\frac{p-1}{2}},&2<p<3\;\text{and }\nu\geq 2
    \end{cases}
\end{equation*}
and $u$ satisfies
\begin{equation*}
    \left(\nu-\frac{1}{2}\right)S(a,b,n)^{\frac{p}{p-2}}\leq\int_{\R^n}|x|^{-2a}|\nabla u|^2\mathrm{d}x\leq \left(\nu+\frac{1}{2}\right)S(a,b,n)^{\frac{p}{p-2}}.
\end{equation*}
Here $D_a^{-1,2}(\R^n)$ denotes the dual space of $D_a^{1,2}(\R^n)$. However, as was noticed in \cite{Den1,Fra1}, there is an inaccuracy in \cite{Wei}: due to the absence of the nondegeneracy condition, the stability results of \cite{Wei} do not hold when $(a,b)$ falls on the Felli-Schneider curve.

\vskip0.1in

We refer the readers to \cite{Ary,Bha,De,Liu,Pic,Zha,Zho} for various quantitative stability results on other Sobolev-type inequalities.

\vskip0.1in

In this paper, we aim to derive appropriate quantitative stability results of the equation \eqref{eq} in the case $a<0,b=b_{\text{FS}}(a)$. There are two main motivations for our work.

\vskip0.1in
The first motivation comes from the fact that Wei and Wu's result \eqref{www} does not cover the Felli-Schneider curve and its proof can hardly be applied to this situation. One of the key tools utilized in their work is the non-degeneracy of the equation \eqref{eq}. Indeed, to the best of our knowledge, all such type quantitative stability results to date are dependent on the non-degeneracy. However, it is well known that (see \cite{Fel,Fra1} for example) the equation \eqref{eq} is non-degenerate when $a\geq 0$ or $a<0,b>b_{\text{FS}}(a)$, but degenerate when $a<0,b=b_{\text{FS}}(a)$. Specifically, consider the linearized equation of \eqref{eq} near some Talenti bubble $U_{\lambda_0}$:
\begin{align}\label{line}
    \div(|x|^{-2a}\nabla v)+(p-1)|x|^{-pb}U_{\lambda_0}^{p-2}v=0.
\end{align}
We use $\Bar{X}_{\lambda_0}$ to denote the set of solutions to \eqref{line}. Clearly $\Bar{X}_{\lambda_0}$ contains the trivial elements $t\partial_\lambda U_{\lambda_0},t\in\R$. We also define $X_{\lambda_0}$ to be the orthogonal complement of $ \text{Span}\{\partial_\lambda U_{\lambda_0}\}$ in $\Bar{X}_{\lambda_0}$ with respect to the $D_a^{1,2}(\R^n)$ scalar product. When $a\geq 0$ or $a<0,b>b_{\text{FS}}(a)$, $X_{\lambda_0}=\emptyset$ and we say \eqref{eq} is non-degenerate. When $a<0$ and $b=b_{\text{FS}}(a)$, $X_{\lambda_0}\neq \emptyset$ and we say \eqref{eq} is degenerate  (see Lemma \ref{le2} for an equivalent formulation). Due to the degeneracy, it is easy to construct examples which do not satisfy \eqref{www}. So a natural and challenging question is: is it possible to replace $F_1$ by another function $F$ such that the estimate \eqref{www} vaild in the case $a<0,b=b_{\text{FS}}(a)$?

\vskip0.1in
The second motivation comes from the quantitative stability of the Caffarelli-Kohn-Nirenberg inequality \eqref{ckn} in the functional setting. The stability of functional inequalities was initially raised by Br\'ezis and Lieb \cite{Bre}. They asked whether the homogeneous Sobolev inequality can be reinforced. This question was answered affirmatively by Bianchi and Egnell \cite{Bia}. After the work \cite{Bia}, the stability of inequalities has drawn attention and has been addressed by lots of researchers. We do not attempt a survey of the extensive literature, but refer the readers to \cite{Che,Dol0,Fig1,Fig2,Fra,Fra1,Wei,Zha,Zho} and the references therein for more background. Similar to the stability in the critical point setting, many quantitative stability results of functional inequalities depend on the non-degeneracy condition. Recently some interesting degenerate stability results were discovered by Engelstein, Neumayer and Spolaor \cite{Eng}. Their results were later improved by Frank \cite{Fra}. 

Let's go back to our setting. In the case $a\geq 0$ or $a<0,b>b_{\text{FS}}(a)$, Wei and Wu \cite{Wei} established the following sharp quadratic estimate
\begin{align}
    \norm*{u}_{D_a^{1,2}(\R^n)}^2-S(a,b,n)\norm*{|x|^{-b}u}_{L^p(\R^n)}^2\geq c(a,b,n)\inf_{s\in\R,\lambda>0}\norm*{u-sU_\lambda}_{D_a^{1,2}(\R^n)}^2.\nonumber
\end{align}
In the case $a<0,b=b_{\text{FS}}(a)$, due to the degeneracy, the above estimate no longer hold. Recently, Frank and Peteranderl \cite{Fra1} studied this case and found the optimal exponent is $4$:
\begin{equation*}
    \norm*{u}_{D_a^{1,2}(\R^n)}^4-S(a,b,n)^2\norm*{|x|^{-b}u}_{L^p(\R^n)}^4\geq c(a,b,n)\inf_{s\in\R,\lambda>0}\norm*{u-sU_\lambda}_{D_a^{1,2}(\R^n)}^4.
\end{equation*}
Moreover, if the infimum is attained by $s_0U_{\lambda_0}$, their proof yields
\begin{align}
    \norm*{u}_{D_a^{1,2}(\R^n)}^2-S(a,b,n)\norm*{|x|^{-b}u}_{L^p(\R^n)}^2\geq c(a,b,n)\norm*{\Pi_{X_{\lambda_0}}^\perp(u)-s_0U_{\lambda_0}}_{D_a^{1,2}(\R^n)}^2,\nonumber
\end{align}
where $\Pi_{X_{\lambda_0}}$ is the orthogonal projection in $D_a^{1,2}(\R^n)$ onto $X_{\lambda_0}$ and $\Pi_{X_{\lambda_0}}^\perp=1-\Pi_{X_{\lambda_0}}$. Since in some sense these two types of stability are closely related, we believe that there exist corresponding degenerate stabilities in the critical point setting.

\vskip0.1in
Let us state our main results.
\begin{theorem}\label{thm1}
    Assume $n\geq 2$, $a<0$, $b=b_{\mathrm{FS}}(a)$ and $p=\frac{2n}{n-2+2(b-a)}$. Then there exists a positive constant $C=C(a,n)$ such that, for any nonnegative function $u\in D_a^{1,2}(\R^n)$ satisfying
    \begin{equation}\label{zzz1}
     \frac{1}{2}S(a,b,n)^{\frac{p}{p-2}}\leq \int_{\R^n}|x|^{-2a}|\nabla u|^2\mathrm{d}x\leq \frac{3}{2}S(a,b,n)^{\frac{p}{p-2}},
    \end{equation}
    it holds that
    \begin{equation}\label{main-re1}
        \inf_{\lambda>0}\norm*{u- U_{\lambda}}_{D_a^{1,2}(\R^n)}^3\leq C\norm*{H(u)}_{D_a^{-1,2}(\R^n)}
    \end{equation}
    The above estimate is sharp in the sense that there exists a nonnegative sequence $\{w_k\}_k\subset D_a^{1,2}(\R^n)$  satisfying \eqref{zzz1} such that $w_k\notin \mathcal{M}_{a,b,n}$,
    \begin{equation*}
        \inf_{\lambda>0}\norm*{w_k- U_{\lambda}}_{D_a^{1,2}(\R^n)}\rightarrow 0
    \end{equation*}
    and
    \begin{equation*}
        \limsup_{k\rightarrow+\infty}\frac{\norm*{H(w_k)}_{D_a^{-1,2}(\R^n)}}{  \inf\limits_{\lambda>0}\norm*{w_k- U_{\lambda}}_{D_a^{1,2}(\R^n)}^3}<+\infty.
    \end{equation*}
    Moreover, if the infimum in \eqref{main-re1} is attained by $U_{\lambda_0}$, then there exists a positive constant $C_1=C_1(a,n)$ such that
    \begin{align}\label{main-re2}
        \norm*{\Pi_{X_{\lambda_0}}^\perp(u)-U_{\lambda_0}}_{D_a^{1,2}(\R^n)}\leq C_1\norm*{H(u)}_{D_a^{-1,2}(\R^n)}
    \end{align}
    provided
    \begin{align}\label{qaz1}
        \norm*{\Pi_{X_{\lambda_0}}^\perp(u)-U_{\lambda_0}}_{D_a^{1,2}(\R^n)}\geq C_1F_2\left(\norm*{\Pi_{X_{\lambda_0}}(u)}_{D_a^{1,2}(\R^n)}\right),
    \end{align}
    where
    \begin{align}
        F_2(x):=\begin{cases}
            x^2,&p\geq3,\\
            x^{p-1},&p<3.
        \end{cases}\nonumber
    \end{align}
\end{theorem}
\begin{theorem}\label{thm2}
    Assume $n\geq 2$, $\nu\geq 2$, $a<0$, $b=b_{\mathrm{FS}}(a)$ and $p=\frac{2n}{n-2+2(b-a)}$. Then there exists a positive constant $C=C(a,\nu,n)$ such that, for any nonnegative function $u\in D_{a}^{1,2}(\R^n)$ satisfying 
    \begin{equation}\label{gg}
     \left(\nu-\frac{1}{2}\right)S(a,b,n)^{\frac{p}{p-2}}\leq\int_{\R^n}|x|^{-2a}|\nabla u|^2\mathrm{d}x\leq \left(\nu+\frac{1}{2}\right)S(a,b,n)^{\frac{p}{p-2}},
    \end{equation}
    if the infimum
    \begin{align}
        \inf_{\substack{s_i\in\R_+\\
    1\leq i\leq\nu}}\norm*{u-\sum_{i=1}^\nu U_{s_i}}_{D_a^{1,2}(\R^n)}\nonumber
    \end{align}
    is attained by $U_{\lambda_i},1\leq i\leq \nu$, then 
    \begin{equation}\label{main-re3}
        \norm*{\Pi_{X}^\perp(u)-\sum_{i=1}^\nu U_{\lambda_i}}_{D_a^{1,2}(\R^n)}\leq CF_1\left(\norm*{H(u)}_{D_a^{-1,2}(\R^n)}\right)
    \end{equation}
    provided
    \begin{align}\label{qaz2}
        \norm*{\Pi_{X}^\perp(u)-\sum_{i=1}^\nu U_{\lambda_i}}_{D_a^{1,2}(\R^n)}\geq CF_1\left(\norm*{\Pi_{X}(u)}^2_{D_a^{1,2}(\R^n)}\right).
    \end{align}
    Here $X:=\mathop{\oplus}\limits_{i=1}^\nu X_{\lambda_i}$.
\end{theorem}
\begin{remark}
    In the single bubble case (Theorem \ref{thm1}), we in fact do not need the nonnegativity of the function $u$. Precisely, for any function $u\in D_a^{1,2}(\R^n)$ satisfying \eqref{zzz1}, it holds that
    \begin{align}
        \inf_{\lambda>0,\theta\in\{-1,1\}}\norm*{u-\theta U_{\lambda}}_{D_a^{1,2}(\R^n)}^3\leq C\norm*{H(u)}_{D_a^{-1,2}(\R^n)}.\nonumber
    \end{align}
    This follows from Theorem \ref{thm3} below  and the fact that for each sign-changing solution $u_0$ to the equation \eqref{eq}, we have the following energy estimate
    \begin{equation*}
        \int_{\R^n}|x|^{-2a}|\nabla u_0|^2\mathrm{d}x\geq 2S(a,b,n)^{\frac{p}{p-2}}.
    \end{equation*}
    This estimate can be proved by testing the equation \eqref{eq} with $(u_0)^+$ and $(u_0)^-$, and applying the inequality \eqref{ckn}.
\end{remark}
\begin{remark}
    In the case $\nu=1$, Theorem \ref{thm1} completes the stability analysis of the Caffarelli-Kohn-Nirenberg inequality within the critical point framework initiated by Wei and Wu \cite{Wei} (see \eqref{www}). It remains unclear whether the cubic estimate \eqref{main-re1} holds for $\nu\geq2$.
\end{remark}
\begin{remark}
    Unlike quantitative stabilities in the functional setting, we do not always have \eqref{main-re2} or \eqref{main-re3}. Indeed, in Section \ref{sec3} we show that, for any $0<\delta\ll 1$, there exists a positive function $u$ (see \eqref{cons}) such that the infimum in \eqref{main-re1} is attained by $U_1$ and 
    \begin{align}
         \norm*{\Pi_{X_1}(u)}_{D_a^{1,2}(\R^n)}\approx \delta,\quad\norm*{\Pi_{X_1}^\perp(u)-U_{1}}_{D_a^{1,2}(\R^n)}\approx \delta^2,\quad \norm*{H(u)}_{D_a^{-1,2}(\R^n)}\approx \delta^3.\nonumber
    \end{align}
    Thus in Theorem \ref{thm1} and Theorem \ref{thm2}, we pose the restrictions \eqref{qaz1} and \eqref{qaz2} which indicate that the function $u$ is not too degenerate. From the example above, we see these restrictions are optimal when $\nu=1,p\geq3$ or $\nu\geq2,p>3$. We wonder whether these restrictions are sharp in other cases. 
\end{remark}
The following result is a direct consequence of the above theorems. To derive it, one just need to note that, for any $u\in D_a^{1,2}(\R^n)$ centrally symmetric and $\lambda>0$, we have $\Pi_{X_\lambda}(u)=0$.
\begin{corollary}
    Assume $n\geq 2$, $\nu\geq 1$, $a<0$, $b=b_{\mathrm{FS}}(a)$ and $p=\frac{2n}{n-2+2(b-a)}$. Then there exists a positive constant $C=C(a,\nu,n)$ such that, for any nonnegative centrally symmetric function $u\in D_{a}^{1,2}(\R^n)$ satisfying \eqref{gg}, it holds that
    \begin{align}
          \inf_{\substack{s_i\in\R_+\\
    1\leq i\leq\nu}}\norm*{u-\sum_{i=1}^\nu U_{s_i}}_{D_a^{1,2}(\R^n)}\leq CF_1\left(\norm*{H(u)}_{D_a^{-1,2}(\R^n)}\right).\nonumber
    \end{align}
    Moreover, the above estimate is sharp.
\end{corollary}
Thanks to the global compactness principle (see \cite[Proposition 3.2]{Wei} for example), to establish Theorem \ref{thm1} and Theorem \ref{thm2}, we only need to prove their local versions.
\begin{theorem}\label{thm3}
    Assume $n\geq 2$, $a<0$, $b=b_{\mathrm{FS}}(a)$ and $p=\frac{2n}{n-2+2(b-a)}$. Then there exist two positive constants $\delta=\delta(a,n)$ and $C=C(a,n)$ such that, for any function $u\in D_a^{1,2}(\R^n)$ satisfying
    \begin{align}
    \norm*{u- U_1}_{D_a^{1,2}(\R^n)}\leq \delta,\nonumber
    \end{align}
    the estimate \eqref{main-re1} holds and is optimal. Moreover, if the infimum is attained by $U_{\lambda_0}$ and \eqref{qaz1} holds, then \eqref{main-re2} holds.
\end{theorem}
\begin{theorem}\label{thm4}
    Assume $n\geq 2$, $\nu\geq2$, $a<0$, $b=b_{\mathrm{FS}}(a)$ and $p=\frac{2n}{n-2+2(b-a)}$. Then there exist two positive constants $\delta=\delta(a,\nu,n)$ and $C=C(a,\nu,n)$ such that, for any function $u\in D_a^{1,2}(\R^n)$ satisfying
    \begin{align}
        \inf_{\substack{s_i\in\R_+\\
    1\leq i\leq\nu}}\norm*{u-\sum_{i=1}^\nu U_{s_i}}_{D_a^{1,2}(\R^n)}\leq \delta\nonumber
    \end{align}
    and the infimum is attained by $U_{\lambda_i},1\leq i\leq \nu$ with 
    \begin{align}
        \min\left\{\frac{\lambda_i}{\lambda_j},\frac{\lambda_j}{\lambda_i}\right\}\leq \delta,\quad \forall\;i\neq j,\nonumber
    \end{align}
    the estimate \eqref{main-re3} holds provided \eqref{qaz2} holds.
\end{theorem}
\begin{remark}
    Note that in these local versions, we do not need $u$ to be positive. Instead, we assume $u$ is close to a sum of Talenti bubbles.
\end{remark}
\begin{remark}
    By further exploiting the proof of Theorem \ref{thm3}, we can obtain the following asymptotic estimate: let $\{\varphi_k\}_k$ be a sequence in $D_a^{1,2}(\R^n)$ such that $\varphi_k\notin \mathcal{M}_{a,b,n}$ and $\inf\limits_{\lambda>0}\norm*{\varphi_k- U_{\lambda}}_{D_a^{1,2}(\R^n)}\rightarrow 0$, then
    \begin{equation}\label{asy}
        \liminf_{k\rightarrow+\infty}\frac{\norm*{H(\varphi_k)}_{D_a^{-1,2}(\R^n)}}{  \inf\limits_{\lambda>0}\norm*{\varphi_k- U_{\lambda}}_{D_a^{1,2}(\R^n)}^3}\geq R(p,n).
    \end{equation}
    Here $R(p,n)$ is a complicated number defined in \eqref{rrr}. Generally, this number is not optimal. We will elaborate on \eqref{asy} in Remark \ref{re3}.
\end{remark}
\vspace{10pt}
Let us briefly address our methods in deriving Theorem \ref{thm3} and Theorem \ref{thm4}.

To establish Theorem \ref{thm3}, we adapt the ideas from \cite{Fra1,Wei}. The estimate \eqref{main-re2} can be obtained by refining the arguments of \cite[Proposition 5.1]{Wei}. As for the estimate \eqref{main-re1}, we first test the equation \eqref{eq} with a suitable function and then take advantage of a expansion technique presented in \cite{Fra,Fra1} to deal with nonlinear terms. Finally we can obtain a unexplicit lower bound of $\norm*{H(u)}_{D_a^{-1,2}(\R^n)}$, which is a quadratic form of $u$. The positivity of this lower bound has already been investigated in \cite{Fra1} and thus \eqref{main-re1} follows.

It is worth mentioning the following three points. The first point is that, unlike the non-degenerate case (\cite[Proposition 5.1]{Wei}), the most canonical choice of the test function does not allow us to estimate the $D_a^{-1,2}(\R^n)$-norm of the equation \eqref{eq} well. The test function we choose in this paper is in some sense optimal. We explain this point further in Remark \ref{re1}. The second point is that we need to expand the equation \eqref{eq} to the fourth order (note that in the non-degenerate case, a second-order expansion is sufficient). Since in general $p<4$, we cannot expand it directly. Here we employ a high order expansion technique from \cite{Fra,Fra1} to overcome this problem. The third point is that, the construction of the functions $\{w_k\}_k$ in Theorem \ref{thm1} is technical. In the non-degenerate case, the optimality can be easily proved. However, in our case, many natural choices of $\{w_k\}_k$ can only yield a quadratic estimate. To achieve the desired cubic estimate, we use the solutions of some special Schr\"odinger equations to perturb certain natural choice of $\{w_k\}_k$. A further explanation of this point can be found in Remark \ref{re2}.
\vspace{10pt}

To establish Theorem \ref{thm4}, we first modify the finite-dimensional reduction method employed in \cite{Wei} to find an approximation function $\rho_0$ of $\rho:=u-\sum\limits_{i=1}^\nu U_{\lambda_i}$. $\rho_0$ has a good pointwise estimate and its $D_a^{1,2}(\R^n)$-norm can be well controlled. Next by applying careful estimates, we manage to work out the relations between $\norm*{\Pi_X(\rho)}_{D_a^{1,2}(\R^n)}$, $\norm*{\Pi_X^\perp(\rho)-\rho_0}_{D_a^{1,2}(\R^n)}$ and $\norm*{H(u)}_{D_a^{-1,2}(\R^n)}$. Under the restriction \eqref{qaz2}, $\norm*{\Pi_X^\perp(\rho)-\rho_0}_{D_a^{1,2}(\R^n)}$ can be controlled by $\norm*{H(u)}_{D_a^{-1,2}(\R^n)}$. Finally, combining above two steps gives the estimate \eqref{main-re3}. 
\vskip0.1in

Our paper is organized as follows. In Section \ref{sec2}, we reformulate our main results on the cylinder $\C=\R\times \S^{n-1}$ and provide some technical lemmas. Concretely speaking, we employ the Emden-Fowler transformation to transform the Caffarelli-Kohn-Nirenberg inequality \eqref{ckn} and the equation \eqref{eq} into certain Gagliardo-Nirenberg interpolation inequality and its corresponding Euler-Lagrange equation on $\C$, respectively. Based on this transformation, Theorem \ref{thm3} and Theorem \ref{thm4} are equivalent to Theorem \ref{thm5} and Theorem \ref{thm6}, respectively. We also list several useful lemmas that play an important role in our study on $\C$. In Section \ref{sec3}, we are devoted to the proofs of Theorem \ref{thm1} and Theorem \ref{thm5}. We also explain some key points of our proofs. In Section \ref{sec4}, we deal with Theorem \ref{thm2} and Theorem \ref{thm6}.

\vskip0.1in

\noindent$\textbf{Notations.}$ Throughout this paper, for simplicity, we consistently use $C$ and $c$ to denote positive quantities depending only on the parameters $n,a,b,p$ and possibly $\nu$. The values of $C$ and $c$ may vary from line to line. Generally, $C$ is used to denote large numbers, while $c$ represents small numbers. For any two quantities $a_1$ and $a_2$, we shall write $a_1\lesssim a_2$ (resp. $a_1\gtrsim a_2$) if $a_1\leq Ca_2$ (resp. $a_1\geq ca_2)$. Also we say $a_1\approx a_2$ if $a_1\lesssim a_2$ and $a_1\gtrsim a_2$. Given any quantity $\varepsilon$, we say $a_1\lesssim_\varepsilon a_2$ if $a_1\lesssim C(\varepsilon)a_2$, where $C(\varepsilon)$ is a positive number depending only on $\varepsilon$. We similarly define $\gtrsim_\varepsilon$ and $\approx_\varepsilon$. For any two sequences $\{A_k\}_k$ and $\{B_k\}_k$ of numbers, we write $A_k=\O(B_k)$ if $|A_k|\leq C|B_k|,\;\forall k\in\N_+$. When referring to a positive variable $A$, we will use the notation $o_A(1)$ to indicate any quantity that tends towards zero as $A$ approaches zero. For any two numbers $x$ and $y$, we define two new numbers
\begin{equation*}
    \chi({x>y}):=\begin{cases}
        1\quad\text{if }x>y\\
        0\quad\text{if }x\leq y,
    \end{cases}\quad\chi({x\geq y}):=\begin{cases}
        1\quad\text{if }x\geq y\\
        0\quad\text{if }x< y.
    \end{cases}
\end{equation*}

\noindent$\textbf{Acknowledgement.}$ The authors are very grateful to Rupert L. Frank and Jonas W. Peteranderl for helpful discussions that improved this paper. We also thank the anonymous referees for their careful reading of the manuscript and for valuable comments that significantly improved the paper.

\section{A reformulation and some technical lemmas}\label{sec2}
\subsection{A reformulation}
Consider the following so-called Emden-Fowler transformation
\begin{equation*}
    u(r,\theta)=r^{\frac{2+2a-n}{2}}v(s,\theta),
\end{equation*}
where $r=|x|$, $s=-\ln{r}$ and $\theta=\frac{x}{r}\in\S^{n-1}$. As was observed in \cite{Cat}, with this transformation, the Caffarelli-Kohn-Nirenberg inequality \eqref{ckn} can be rewritten as
\begin{equation}\label{ckn1}
    S(a,b,n)\norm*{v}^2_{L^p(\C)}\leq \norm*{\partial_s v}^2_{L^2(\C)}+\norm*{\nabla_\theta v}^2_{L^2(\C)}+\Lambda\norm*{v}^2_{L^2(\C)}
\end{equation}
and the critical Hardy-H\'enon equation \eqref{eq} is equivalent to
\begin{equation}\label{eq1}
    H_1(v):=\partial_s^2 v+\Delta_\theta v-\Lambda v+|v|^{p-2}v=0.
\end{equation}
Here $\Lambda:=\left(\frac{n-2-2a}{2}\right)^2$ and $\nabla_\theta,\Delta_\theta$ denote the gradient and the Laplace-Beltrami operator on $\S^{n-1}$, respectively. This transformation provides an isometry between the spaces $D_a^{1,2}(\R^n)$ and $H^1(\C)$. Here we equip $H^1(\C)$ with the following norm
\begin{equation*}
    \norm*{v}_{H^1(\C)}=\left(\norm*{\partial_s v}^2_{L^2(\C)}+\norm*{\nabla_\theta v}^2_{L^2(\C)}+\Lambda\norm*{v}^2_{L^2(\C)}\right)^{\frac{1}{2}}.
\end{equation*}
Note that in this case, the Felli-Schneider curve can be parametrized by $n\geq2$, $2<p<2^*$ and $\Lambda=\frac{4(n-1)}{p^2-4}$. Along the Felli-Schneider curve, the set of positive solutions to the equation \eqref{eq1} is given by
\begin{equation}\label{sol1}
    \mathcal{M}_{p,\Lambda,n}=\left\{V_t(s)\;\Big|\;V_t(s):=\beta(\cosh{(\alpha(s-t))})^{-\frac{2}{p-2}},t\in\R\right\},
\end{equation}
where $\alpha:=\frac{p-2}{2}\sqrt{\Lambda}$ and $\beta:=\left(\frac{p\Lambda}{2}\right)^{\frac{1}{p-2}}$. The elements in $\M_{p,\Lambda,n}$ are also called \textbf{Talenti bubbles}. Similar to the original setting on $\R^n$, for any $t\in \R$, we define $Y_t=\text{Span}\{V_t^{\frac{p}{2}}\theta_i,1\leq i\leq n\}$. From Lemma \ref{le2}, $Y_t\subset H^1(\C)$ consists of nontrivial solutions to the linearized equation of \eqref{eq1} near the Talenti bubble $V_t$. We also define $\Pi_{Y_t}$ and $\Pi_{Y_t}^\perp$ to be the projections onto $Y_t$ and the orthogonal complement of $Y_t$ in $H^1(\C)$, respectively.

Now we can state the following exact counterparts of Theorem \ref{thm3} and Theorem \ref{thm4} on $\C$.
\begin{theorem}\label{thm5}
    Assume $n\geq2$, $2<p<2^*$ and $\Lambda=\frac{4(n-1)}{p^2-4}$. Then there exist two positive constants $\delta=\delta(p,n)$ and $C=C(p,n)$ such that, for any function $v\in H^1(\C)$ satisfying
    \begin{equation}
        \norm*{v- V_0}_{H^1(\C)}\leq \delta,\nonumber
    \end{equation}
    it holds that
    \begin{equation}\label{main-re4}
        \inf_{t\in\R}\norm*{v- V_t}_{H^1(\C)}^3\leq C\norm*{H_1(v)}_{H^{-1}(\C)}.
    \end{equation}
    The above cubic estimate is sharp. Moreover, if the infimum in \eqref{main-re4} is attained by $V_{t_0}$, then there exists a positive constant $C_1=C_1(p,n)$ such that
    \begin{align}\label{main-re5}
        \norm*{\Pi_{Y_{t_0}}^\perp(v)- V_{t_0}}_{H^1(\C)}\leq C\norm*{H_1(v)}_{H^{-1}(\C)}
    \end{align}
    provided
    \begin{align}
        \norm*{\Pi_{Y_{t_0}}^\perp(v)- V_{t_0}}_{H^1(\C)}\geq C_1F_2\left(\norm*{\Pi_{Y_{t_0}}(v)}_{H^1(\C)}\right).\nonumber
    \end{align}
\end{theorem}
\begin{theorem}\label{thm6}
    Assume $n\geq2$, $\nu\geq2$, $2<p<2^*$ and $\Lambda=\frac{4(n-1)}{p^2-4}$. Then there exist two positive constants $\delta=\delta(p,\nu,n)$ and $C=C(p,\nu,n)$ such that, for any function $v\in H^1(\C)$ satisfying
    \begin{align}
        \inf_{\substack{s_i\in\R\\
    1\leq i\leq\nu}}\norm*{v-\sum_{i=1}^\nu V_{s_i}}_{H^1(\C)}\leq \delta\nonumber
    \end{align}
    and the infimum is attained by $V_{t_i},1\leq i\leq\nu$ with
    \begin{align}
        |t_i-t_j|\geq \delta^{-1},\quad\forall\;i\neq j,\nonumber
    \end{align}
    the following estimate holds
    \begin{align}\label{main-re6}
        \norm*{\Pi_{Y}^\perp(v)-\sum_{i=1}^\nu V_{t_i}}_{H^1(\C)}\leq CF_1\left(\norm*{H_1(v)}_{H^{-1}(\C)}\right)
    \end{align}
    provided
    \begin{align}
        \norm*{\Pi_{Y}^\perp(v)-\sum_{i=1}^\nu V_{t_i}}_{H^1(\C)}\geq CF_1\left(\norm*{\Pi_{Y}(v)}^2_{H^{1}(\C)}\right).\nonumber
    \end{align}
    Here $Y:=\mathop{\oplus}\limits_{i=1}^\nu Y_{t_i}$.
\end{theorem}
In the following two sections we focus on the proofs of the above two new theorems instead of the two original ones. It seems that dealing with the equation \eqref{eq1} on $\C$ can simplify some expressions and computations.

\subsection{Some technical lemmas}
In this subsection we provide some important lemmas. The first lemma contains three elementary inequalities.
\begin{lemma}\label{le1}
    Assume $x,y\in\R$ and $p>2$. Then we have
    \begin{equation*}
        \left||x+y|^{p-2}(x+y)-|x|^{p-2}x-(p-1)|x|^{p-2}y\right|\lesssim \chi(p>3)|x|^{p-3}y^2+|y|^{p-1}
    \end{equation*}
    and
    \begin{align}
        \left||x+y|^{p-2}-|x|^{p-2}\right|\cdot |x|\lesssim\begin{cases}
            |x||y|^{p-2}+|x|^{p-2}|y|,&p\geq 3,\\
            |xy|^{\frac{p-1}{2}},&p<3.
        \end{cases}\nonumber 
    \end{align}
    Furthermore, if $|y|\leq\frac{1}{2}|x|$, we have
    \begin{equation*}
    \begin{aligned}
        \Big|&|x+y|^{p-2}(x+y)-|x|^{p-2}x-(p-1)|x|^{p-2}y-\frac{(p-1)(p-2)}{2}|x|^{p-3}y^2\\
        &-\frac{(p-1)(p-2)(p-3)}{6}|x|^{p-4}y^3\Big|\lesssim |x|^{p-5}y^4.
    \end{aligned}
    \end{equation*}
\end{lemma}
The first and third inequalities are standard. The second inequality can be obtained by treating the cases $2|x|\leq |y|$ and $2|x|\geq|y|$ separately. The second lemma concerns with the spectrum of the linearized equation of \eqref{eq1} near certain Talenti bubble. We refer the readers to \cite{Fel,Fra1} for the proof of this Lemma.
\begin{lemma}\label{le2}
    Assume $n\geq2$, $2<p<2^*$ and $\Lambda=\frac{4(n-1)}{p^2-4}$. Denote by $\{\gamma_i\}_i$ the eigenvalues of the equation
    \begin{equation*}
        -\partial_s^2 v-\Delta_\theta v+\Lambda v=\gamma V_0^{p-2}v,\quad v\in H^1(\C),
    \end{equation*}
    then one has
    \begin{equation*}
        \gamma_1=1,\quad\gamma_2=p-1,\quad\gamma_3>p-1.
    \end{equation*}
    Moreover, the first and second eigenspaces are given by $$\mathrm{span}\{V_0\}\;\; \mathrm{and} \;\;\mathrm{span}\left\{\partial_s V_0;V_0^{\frac{p}{2}}\theta_i,1\leq i\leq n\right\},$$ respectively. Here $\theta_i\;(1\leq i\leq n)$ denote the Cartesian coordinates restricted to $\S^{n-1}$.
\end{lemma}
The degeneracy of \eqref{eq1} (and thus \eqref{eq}) is caused by the nontrivial elements $V_0^{\frac{p}{2}}\theta_i,1\leq i\leq n$. Note that via the Emden-Fowler transformation, $u\in D_a^{1,2}(\R^n)$ is centrally symmetric if and only if $v\in H^1(\C)$ satisfies $v(s,\theta)=v(s,-\theta)$. Obviously for such $v$ one has $\Pi_{Y_t}(v)=0,\forall t\in\R$. Therefore, $\Pi_{X_\lambda}(u)=0,\forall \lambda>0$.
The third lemma concerns the power series solution of a special Schr\"odinger equation.
\begin{lemma}\label{le3}
    Assume $n\geq2$, $2<p<2^*$ and $\Lambda=\frac{4(n-1)}{p^2-4}$. Then there exist $\tau,\eta\in\R$ and $\{A_k\}_k,\{B_k\}_k\subset\R$, depending only on $p$ and $n$, such that the unique $H^1(\R)$-solution $g$ of
    \begin{equation*}
        -\partial_s^2 g+(2n+\Lambda)g-(p-1)V_0^{p-2}g=V_0^{2p-3}
    \end{equation*}
    has the following power series representation
    \begin{equation*}
        g(s)=\tau V_0^{\sqrt{1+\frac{2n}{\Lambda}}}\sum_{k=0}^\infty A_k\cosh^{-2k}(\alpha s)+\eta V_0^{2p-3}\sum_{k=0}^\infty B_k\cosh^{-2k}(\alpha s).
    \end{equation*}
    Here $\alpha:=\frac{p-2}{2}\sqrt{\Lambda}$. Moreover,
    \begin{equation*}
        A_k=\O(k^{-\frac{3}{2}}),\quad B_k=\O(k^{-\frac{3}{2}}).
    \end{equation*}
\end{lemma}
The uniqueness is a direct consequence of spectral analysis, and the power series representation can be calculated concretely. We refer the readers to \cite[Lemma 12]{Fra1} for its detailed proof. The fourth lemma comes from the discussions in \cite[Section 4]{Fra1}.
\begin{lemma}\label{le4}
    Assume $n\geq2$, $2<p<2^*$ and $\Lambda=\frac{4(n-1)}{p^2-4}$. For any $\varepsilon$ small, set
    \begin{equation*}
        E_\varepsilon:=\inf_{g\in \hat{Y}}\left((1-\varepsilon)\norm*{g}^2_{H^1(\C)}-(p-1)\int_{\C}\left(V_0^{p-2}g^2+(p-2)V_0^{2p-3}\theta_n^2g\right)\right)
    \end{equation*}
    and
    \begin{equation*}
        F:=\frac{(p-1)(p-2)}{4}\left(\frac{p-1}{\norm*{V_0}^p_{L^p(\C)}}\left(\int_\C V_0^{2p-2}\theta_n^2\right)^2-\frac{p-3}{3}\int_\C V_0^{3p-4}\theta_n^4\right).
    \end{equation*}
    Here the space $\hat{Y}$ is the orthogonal complement of $\mathrm{Span}\left\{V_0;\partial_s V_0;V_0^{\frac{p}{2}}\theta_i, 1\leq i\leq n\right\}$ in $H^1(\C)$. Then we have
    \begin{equation*}
        F>0,\quad E_0+F>0,\quad E_\varepsilon=E_0+\O(\varepsilon).
    \end{equation*}
    Moreover, when $n\geq 3$,
    \begin{equation*}
        \frac{E_0}{F}+1\rightarrow 0\quad\text{as }p\rightarrow 2^*.
    \end{equation*}
\end{lemma}
Although a detailed proof of this lemma has been given in \cite{Fra1}, we briefly explain the main ideas here for the convenience of the readers, as this lemma is crucial in our proof of Theorem \ref{thm5}. The value of $F$ can be explicitly computed. As for the term $E_\varepsilon$, we expand the function $g$ in spherical harmonics: $g(s,\theta):=\sum\limits_{i,j}a_{i,j}(s)Y_{i,j}(\theta)$. Here $\{Y_{i,j}\}_j$ is an $L^2(\S^{n-1})$-orthonormal basis of spherical harmonics of degree $i$. We take $Y_{0,0}=c_1$ and $Y_{2,0}=c_2(\theta_n^2-\frac{1}{n})$, where $c_1,c_2$ are two normalization factors. Substituting this expansion into $E_\varepsilon$ and employing some spectrum conditions, we can transform $E_\varepsilon$ into a sum of two quadratic forms involving $a_{0,0}$ and $a_{2,0}$. Again, using the spectrum conditions, we can show that these two forms are non-degenerate and can thus be minimized by the standard `completing the square' argument. It is worth mentioning that in this procedure one needs to explicitly solve some special Schr\"odinger equations (see Lemma \ref{le3} for example). Even after obtaining a concrete expression for $E_\varepsilon$, additional work is required to verify that $E_0+F>0$.

Finally, we give some integral estimates involved Talenti bubbles and their derivatives.
\begin{lemma}\label{le5}
    Assume $t_1,t_2\in\R$, $q_1,q_2\geq 0$ and $q_1+q_2=p$. Then for any $\varepsilon>0$, if $|q_1-q_2|\geq\varepsilon$, it holds that
    \begin{align}\label{int1}
        \int_{\C}V_{t_1}^{q_1}V_{t_2}^{q_2}\approx_\varepsilon e^{-\sqrt{\Lambda}|t_1-t_2|\min\{q_1,q_2\}}.
    \end{align}
    When $q_1=q_2=\frac{p}{2}$, it holds that
    \begin{align}\label{int2}
       \int_{\C}V_{t_1}^{\frac{p}{2}}V_{t_2}^{\frac{p}{2}}\approx (|t_1-t_2|+1)e^{-\frac{p\sqrt{\Lambda}}{2}|t_1-t_2|}.
    \end{align}
\end{lemma}
\begin{lemma}\label{le6}
    Given $t_1,t_2\in\R$. If $t_1\ll t_2$, it holds that
    \begin{equation}\label{int3}
        \int_{\C}V_{t_1}^{p-1}\partial_s V_{t_2}\approx e^{\sqrt{\Lambda}(t_1-t_2)}.
    \end{equation}
\end{lemma}
The proof of Lemma \ref{le5} is similar to \cite[Lemma 5.1]{Wei}. The estimate \eqref{int3} can be easily computed using the expressions of Talenti bubbles $V_t,t\in\R$.

\section{Proofs of Theorem \ref{thm1} and Theorem \ref{thm5}}\label{sec3}
In this section we first prove Theorem \ref{thm5}. Theorem \ref{thm3} then follows from the Emden-Fowler transformation. Combining Theorem \ref{thm3} and the global compactness principle \cite[Proposition 3.2]{Wei}, we obtain Theorem \ref{thm1}. We also give some remarks about our proof in the end of this section.
\begin{proof}[Proof of Theorem \ref{thm5}]
    Let us prove \eqref{main-re4} first. If \eqref{main-re4} does not hold for any $\delta>0$, then there exists a sequence $\{v_k\}_k$ of functions in $H^1(\C)$ such that
    \begin{equation}\label{qq1}
        v_k\notin \mathcal{M}_{p,\Lambda,n},\quad \norm*{v_k- V_0}_{H^1(\C)}\rightarrow 0\;\;\text{as }k\rightarrow+\infty
    \end{equation}
    and
    \begin{equation}\label{qq2}
         \frac{\norm*{H_1(v_k)}_{H^{-1}(\C)}}{  \inf\limits_{t\in\R}\norm*{v_k- V_t}_{H^1(\C)}^3}\rightarrow 0\;\;\text{as }k\rightarrow+\infty.
    \end{equation}
    Up to a suitable translation on $\R$, we can assume
    \begin{equation}\label{qq3}
        \inf_{\alpha,t\in\R}\norm*{v_k- \alpha V_t}_{H^1(\C)}=\norm*{v_k- \alpha_k V_0}_{H^1(\C)}.
    \end{equation}
    Set $\rho_k:=v_k-\alpha_kV_0$. From \eqref{qq1} it is easy to see that $|\alpha_k-1|\rightarrow 0$ and $\norm*{\rho_k}_{H^1(\C)}\rightarrow 0$. Since $\alpha_kV_0$ minimizes the distance, by variation we can derive the following two orthogonality conditions:
    \begin{equation}\label{or1}
        \left\langle \rho_k,V_0\right\rangle_{H^1(\C)}=\left\langle \rho_k,\partial_s V_0\right\rangle_{H^1(\C)}=0.
    \end{equation}
    Since $H_1(V_0)=0$, we can rewrite \eqref{or1} as
    \begin{equation}\label{or2}
        \left\langle \rho_k,V^{p-1}_0\right\rangle_{L^2(\C)}=\left\langle \rho_k,V_0^{p-2}\partial_s V_0\right\rangle_{L^2(\C)}=0.
    \end{equation}
Consider the following identity:
    \begin{equation}\label{id}
        \begin{aligned}
            \partial_s^2 v_k+\Delta_\theta v_k-\Lambda v_k+|v_k|^{p-2}v_k=&\;\partial_s^2 \rho_k+\Delta_\theta \rho_k-\Lambda \rho_k+\left(\alpha_k^{p-1}-\alpha_k\right)V_0^{p-1}\\
            &+\left(|v_k|^{p-2}v_k-\alpha_k^{p-1}V_0^{p-1}\right).
        \end{aligned}
    \end{equation}
    Testing it with $V_0$, exploiting Lemma \ref{le1} and using the orthogonality conditions in \eqref{or1} and \eqref{or2}, we can deduce
    \begin{align}
        \left|\alpha_k^{p-1}-\alpha_k\right|\int_\C V_0^{p}\lesssim&\;\norm*{H_1(v_k)}_{H^{-1}(\C)}+\chi(p>3)\int_{\C}V_0^{p-2}|\rho_k|^2+\int_{\C}V_0|\rho_k|^{p-1}\nonumber\\
        \lesssim&\;\norm*{H_1(v_k)}_{H^{-1}(\C)}+\norm*{\rho_k}^{\min\{2,p-1\}}_{H^1(\C)}.\nonumber
        \end{align}
    Since $|\alpha_k-1|\rightarrow 0$, we find
    \begin{equation}\label{qq4}
        |\alpha_k-1|\lesssim \norm*{H_1(v_k)}_{H^{-1}(\C)}+\norm*{\rho_k}^{\min\{2,p-1\}}_{H^1(\C)}.
    \end{equation}
    The combination of \eqref{qq1}, \eqref{qq2}, \eqref{qq3} and \eqref{qq4} indicates
    \begin{equation}\label{qq5}
        \frac{\norm*{H_1(v_k)}_{H^{-1}(\C)}}{  \norm*{\rho_k}_{H^1(\C)}^3}\rightarrow 0\;\;\text{as }k\rightarrow+\infty.
    \end{equation}
In the following we aim to show that \eqref{qq5} is impossible. Up to a suitable rotation on $\S^{n-1}$, we assume
    \begin{equation*}
        \rho_k=\mu_k\left(\hat{\mu}_kV_0^{\frac{p}{2}}\theta_n+\eta_k\right),
    \end{equation*}
    where $\{\mu_k\}_k,\{\hat{\mu}_k\}_k$ are two  sequences of numbers such that
    \begin{equation*}
        (\mu_k,\hat{\mu}_k)=\begin{cases}
            \left(\left\langle \rho_k,V^\frac{p}{2}_0\theta_n\right\rangle_{H^1(\C)}\norm*{V_0^\frac{p}{2}\theta_n}_{H^1(\C)}^{-2},1\right)&\text{if }\left\langle \rho_k,V^\frac{p}{2}_0\theta_n\right\rangle_{H^1(\C)}\neq 0,\\
            \left(\norm*{\rho_k}_{H^1(\C)},0\right)&\text{if }\left\langle \rho_k,V^\frac{p}{2}_0\theta_n\right\rangle_{H^1(\C)}=0;
        \end{cases}
    \end{equation*}
    and $\eta_k\in H^1(\C)$ satisfies
    \begin{equation}\label{or3}
         \left\langle \eta_k,V_0\right\rangle_{H^1(\C)}=\left\langle \eta_k,\partial_s V_0\right\rangle_{H^1(\C)}=\left\langle \eta_k,V_0^\frac{p}{2}\theta_i\right\rangle_{H^1(\C)}=0\quad\text{for any }1\leq i\leq n.
    \end{equation}
    Thanks to Lemma \ref{le2}, the relations in \eqref{or3} can be rewritten as
    \begin{equation}\label{or4}
         \left\langle \eta_k,V^{p-1}_0\right\rangle_{L^2(\C)}=\left\langle \eta_k,V_0^{p-2}\partial_s V_0\right\rangle_{L^2(\C)}=\left\langle \eta_k,V_0^{\frac{3}{2}p-2}\theta_i\right\rangle_{L^2(\C)}=0\quad\text{for any }1\leq i\leq n.
    \end{equation}
    Without loss of generality, we suppose
    \begin{equation}\label{lim}
        \lim_{k\rightarrow+\infty}\norm*{\eta_k}_{H^1(\C)}=L\in [0,+\infty].
    \end{equation}
    We shall divide our analysis for \eqref{qq5} into two parts that are devoted to the case $L>0$ and $L=0$, respectively.
    \\[8pt]
     \emph{Case 1: $L>0$}.

     Let us start by testing $H_1(v_k)$ against $\rho_k$. Using Lemma \ref{le1} and the orthogonality conditions in \eqref{or1} and \eqref{or2} yields
     \begin{align}\label{asd1}
        \norm*{\rho_k}^2_{H^1(\C)}=&-\left\langle H_1(v_k),\rho_k\right\rangle_{L^2(\C)}+\int_{\C}|v_k|^{p-2}v_k\rho_k\nonumber\\
        \leq&\;\norm*{H_1(v_k)}_{H^{-1}(\C)}\norm*{\rho_k}_{H^1(\C)}+(p-1)\alpha_k^{p-2}\int_{\C}V_0^{p-2}\rho^2_k\\
        &+C\chi(p>3)\int_{\C}V_0^{p-3}|\rho_k|^3+C\int_{\C}|\rho_k|^{p}.\nonumber
     \end{align}
     Exploiting the H\"older inequality and the estimate \eqref{qq4}, we obtain
     \begin{align}\label{rr1}
             \norm*{\rho_k}^2_{H^1(\C)}-(p-1+o_\delta(1))\int_{\C}V_0^{p-2}\rho^2_k\lesssim \norm*{H_1(v_k)}_{H^{-1}(\C)}\norm*{\rho_k}_{H^1(\C)}.
     \end{align}
     Thanks to the orthogonality conditions \eqref{or3}, \eqref{or4} and Lemma \ref{le2}, we find
     \begin{align}\label{rr2}
         \norm*{\rho_k}^2_{H^1(\C)}=&\;\mu_k^2\hat{\mu}_k^2\norm*{V_0^{\frac{p}{2}}\theta_n}^2_{H^1(\C)}+\mu_k^2\norm*{\eta_k}^2_{H^1(\C)}\nonumber\\
         \geq&\;(p-1)\mu_k^2\hat{\mu}_k^2\int_{\C}V_0^{p-2}\left(V_0^{\frac{p}{2}}\theta_n\right)^2+(p-1)\mu_k^2\int_{\C}V_0^{p-2}\eta_k^2+\mu_k^2\left(1-\frac{p-1}{\gamma_3}\right)\norm*{\eta_k}^2_{H^1(\C)}\nonumber\\
         =&\;(p-1)\int_{\C}V_0^{p-2}\rho^2_k+\mu_k^2\left(1-\frac{p-1}{\gamma_3}\right)\norm*{\eta_k}^2_{H^1(\C)}.
     \end{align}
     Recall that $\gamma_3>p-1$ is the third eigenvalue defined in Lemma \ref{le2}. Since $L>0$, we have
     \begin{equation}\label{rr3}
         \norm*{\rho_k}_{H^1(\C)}\lesssim_L \mu_k\norm*{\eta_k}_{H^1(\C)}.
     \end{equation}
     Substituting \eqref{rr2} and \eqref{rr3} into \eqref{rr1} gives
     \begin{equation*}
         \norm*{\rho_k}_{H^1(\C)}\lesssim_L \norm*{H_1(v_k)}_{H^{-1}(\C)},
     \end{equation*}
     which obviously contradicts \eqref{qq5}.
     \\[8pt]
     \emph{Case 2: $L=0$}.

     In this case we have $\hat{\mu}_k\equiv 1$ and $\mu_k\neq 0$. Since the $H^1(\C)$-norm of $\rho_k$ is comparable to $\mu_k$, \eqref{qq5} is equivalent to
     \begin{equation}\label{qq6}
         |\mu_k|^{-3}\norm*{H_1(v_k)}_{H^{-1}(\C)}\rightarrow 0\;\;\text{as }k\rightarrow+\infty.
     \end{equation}
      Let us test $H_1(v_k)$ against $\mu_k\left(V_0^\frac{p}{2}\theta_n+2\eta_k\right)$. Exploiting the orthogonality conditions in \eqref{or3}, we deduce
      \begin{align}\label{tt1}
          &\mu_k^2\norm*{V_0^\frac{p}{2}\theta_n}^2_{H^1(\C)}+2\mu_k^2\norm*{\eta_k}^2_{H^1(\C)}\nonumber\\
          =&-\mu_k\left\langle H_1(v_k),V_0^\frac{p}{2}\theta_n+2\eta_k\right\rangle_{L^2(\C)}+\mu_k\int_{\C}|v_k|^{p-2}v_k\left(V_0^\frac{p}{2}\theta_n+2\eta_k\right)\nonumber\\
          \leq&\;C|\mu_k|\norm*{H_1(v_k)}_{H^{-1}(\C)}+\mu_k\int_{\C}|v_k|^{p-2}v_k\left(V_0^\frac{p}{2}\theta_n+2\eta_k\right).
      \end{align}
      To control the second term, we make use of a high order expansion technique presented in \cite{Fra,Fra1}. To be specific, we split the integration domain $\C$ into two parts: $\left\{|\eta_k|<V_0^{\frac{p}{2}}\right\}$ and $\left\{|\eta_k|\geq V_0^\frac{p}{2}\right\}$. In the first part, using the facts that $V_0\in L^\infty(\C)$ and $p>2$, we know $\left|V_0^\frac{p}{2}\theta_n+\eta_k\right|\leq 2V_0^\frac{p}{2}\lesssim V_0$. Therefore, we can apply the third estimate in Lemma \ref{le1} to derive
      \begin{align}\label{tt2}
              |v_k|^{p-2}v_k=&\;\left(\alpha_k V_0\right)^{p-1}+\mu_k(p-1)\left(\alpha_k V_0\right)^{p-2}\left(V_0^\frac{p}{2}\theta_n+\eta_k\right)\nonumber\\
              &+\mu_k^2\frac{(p-1)(p-2)}{2}\left(\alpha_k V_0\right)^{p-3}\left(V_0^\frac{p}{2}\theta_n+\eta_k\right)^2\\
              &+\mu_k^3\frac{(p-1)(p-2)(p-3)}{6}\left(\alpha_k V_0\right)^{p-4}\left(V_0^\frac{p}{2}\theta_n+\eta_k\right)^3+\O\left(\mu_k^4V_0^{3p-5}\right).\nonumber
        \end{align}
      Note that
        \begin{align}
              V_0^{p-2}\left(V_0^\frac{p}{2}\theta_n+\eta_k\right)\left(V_0^\frac{p}{2}\theta_n+2\eta_k\right)=&\;V_0^{p-2}\left(V_0^\frac{p}{2}\theta_n\right)^2+2V_0^{p-2}\eta_k^2+3V_0^{p-2}\left(V_0^\frac{p}{2}\theta_n\right)\eta_k,\nonumber\\
              V_0^{p-3}\left(V_0^\frac{p}{2}\theta_n+\eta_k\right)^2\left(V_0^\frac{p}{2}\theta_n+2\eta_k\right)=&\;V_0^{p-3}\left(V_0^\frac{p}{2}\theta_n\right)^3+4V_0^{p-3}\left(V_0^\frac{p}{2}\theta_n\right)^2\eta_k\nonumber\\
              &+5V_0^{p-3}\left(V_0^\frac{p}{2}\theta_n\right)\eta_k^2+2V_0^{p-3}\eta_k^3\nonumber\\
              =&\;V_0^{p-3}\left(V_0^\frac{p}{2}\theta_n\right)^3+4V_0^{p-3}\left(V_0^\frac{p}{2}\theta_n\right)^2\eta_k+\O\left(\eta_k^2\right),\nonumber\\
                V_0^{p-4}\left(V_0^\frac{p}{2}\theta_n+\eta_k\right)^3\left(V_0^\frac{p}{2}\theta_n+2\eta_k\right)=&\;V_0^{p-4}\left(V_0^\frac{p}{2}\theta_n\right)^4+5V_0^{p-4}\left(V_0^\frac{p}{2}\theta_n\right)^3\eta_k\nonumber\\
                &+9V_0^{p-3}\left(V_0^\frac{p}{2}\theta_n\right)^2\eta_k^2+7V_0^{p-3}\left(V_0^\frac{p}{2}\theta_n\right)\eta_k^3\nonumber\\
                &+2V_0^{p-4}\eta_k^4\nonumber\\
                =&\;V_0^{p-4}\left(V_0^\frac{p}{2}\theta_n\right)^4+\O\left(V_0^{\frac{5}{2}p-4}|\eta_k|\right).\nonumber
            \end{align}
        Substituting these three identities into \eqref{tt2} gives
        \begin{align}\label{tt3}
            &\;|v_k|^{p-2}v_k\left(V_0^\frac{p}{2}\theta_n+2\eta_k\right)\nonumber\\
            =&\;\left(\alpha_k V_0\right)^{p-1}\left(V_0^\frac{p}{2}\theta_n+2\eta_k\right)+\mu_k(p-1)\left(\alpha_k V_0\right)^{p-2}\left(V_0^\frac{p}{2}\theta_n\right)^2\nonumber\\
            &+2\mu_k(p-1)\left(\alpha_k V_0\right)^{p-2}\eta_k^2+3\mu_k(p-1)\left(\alpha_k V_0\right)^{p-2}V_0^\frac{p}{2}\theta_n\eta_k\nonumber\\
            &+\mu_k^2\frac{(p-1)(p-2)}{2}\left(\alpha_k V_0\right)^{p-3}\left(V_0^\frac{p}{2}\theta_n\right)^3+2\mu_k^2(p-1)(p-2)\left(\alpha_k V_0\right)^{p-3}\left(V_0^\frac{p}{2}\theta_n\right)^2\eta_k\nonumber\\
            &+\mu_k^3\frac{(p-1)(p-2)(p-3)}{6}\left(\alpha_k V_0\right)^{p-4}\left(V_0^\frac{p}{2}\theta_n\right)^4\\
            &+\O\left(\mu_k^2\eta_k^2+|\mu_k|^3V_0^{\frac{5}{2}p-4}|\eta_k|+\mu_k^4V_0^{3p-5}\right).\nonumber
        \end{align}
        In the second part, since $\left|V_0^\frac{p}{2}\theta_n+\eta_k\right|\leq 2|\eta_k|$ and $p>2$, we can use the first estimate in Lemma \ref{le1} to deduce
        \begin{align}
                |v_k|^{p-2}v_k=&\;\left(\alpha_k V_0\right)^{p-1}+\mu_k(p-1)\left(\alpha_k V_0\right)^{p-2}\left(V_0^\frac{p}{2}\theta_n+\eta_k\right)\nonumber\\
                &+\O\left(\chi(p>3)\mu_k^2V_0^{p-3}\left(V_0^\frac{p}{2}\theta_n+\eta_k\right)^2+\mu_k^{p-1}\left|V_0^\frac{p}{2}\theta_n+\eta_k\right|^{p-1}\right)\nonumber\\
                =&\;\left(\alpha_k V_0\right)^{p-1}+\mu_k(p-1)\left(\alpha_k V_0\right)^{p-2}\left(V_0^\frac{p}{2}\theta_n+\eta_k\right)\nonumber\\
                &+\O\left(\chi(p>3)\mu_k^2\eta_k^2+|\mu_k|^{p-1}|\eta_k|^{p-1}\right)\nonumber
        \end{align}
        and thus
        \begin{align}\label{tt4}
            &\;|v_k|^{p-2}v_k\left(V_0^\frac{p}{2}\theta_n+2\eta_k\right)\nonumber\\
            =&\;\left(\alpha_k V_0\right)^{p-1}\left(V_0^\frac{p}{2}\theta_n+2\eta_k\right)+\mu_k(p-1)\left(\alpha_k V_0\right)^{p-2}\left(V_0^\frac{p}{2}\theta_n\right)^2\\
            &+2\mu_k(p-1)\left(\alpha_k V_0\right)^{p-2}\eta_k^2+3\mu_k(p-1)\left(\alpha_k V_0\right)^{p-2}V_0^\frac{p}{2}\theta_n\eta_k\nonumber\\
            &+\O\left(\chi(p>3)\mu_k^2|\eta_k|^3+|\mu_k|^{p-1}|\eta_k|^p\right).\nonumber
        \end{align}
        Note that in this part, we also have
        \begin{align}
            \mu_k^2\frac{(p-1)(p-2)}{2}\left(\alpha_k V_0\right)^{p-3}\left(V_0^\frac{p}{2}\theta_n\right)^3=&\;\O\left(\mu_k^2\eta_k^2\right),\nonumber\\
            2\mu_k^2(p-1)(p-2)\left(\alpha_k V_0\right)^{p-3}\left(V_0^\frac{p}{2}\theta_n\right)^2\eta_k=&\;\O\left(\mu_k^2\eta_k^2\right),\nonumber\\
            \mu_k^3\frac{(p-1)(p-2)(p-3)}{6}\left(\alpha_k V_0\right)^{p-4}\left(V_0^\frac{p}{2}\theta_n\right)^4=&\;\O\left(|\mu_k|^3\eta_k^2\right).\nonumber
        \end{align}
        Combining these three identities with \eqref{tt4} shows that the pointwise estimate \eqref{tt3} still holds in this part, but with an additional error term $\O\left(\chi(p>3)\mu_k^2|\eta_k|^3+|\mu_k|^{p-1}|\eta_k|^p\right)$.

        Exploiting our arguments above and utilizing the orthogonality conditions in \eqref{or3} and \eqref{or4} yields
        \begin{align}\label{ff1}
            &\int_{\C}|v_k|^{p-2}v_k\left(V_0^\frac{p}{2}\theta_n+2\eta_k\right)\nonumber\\
            =&\;\mu_k(p-1)\int_\C\left(\alpha_k V_0\right)^{p-2}\left(V_0^\frac{p}{2}\theta_n\right)^2+2\mu_k(p-1)\int_\C\left(\alpha_k V_0\right)^{p-2}\eta_k^2\nonumber\\
            &+2\mu_k^2(p-1)(p-2)\int_\C\left(\alpha_k V_0\right)^{p-3}\left(V_0^\frac{p}{2}\theta_n\right)^2\eta_k\nonumber\\
            &+\mu_k^3\frac{(p-1)(p-2)(p-3)}{6}\int_\C\left(\alpha_k V_0\right)^{p-4}\left(V_0^\frac{p}{2}\theta_n\right)^4\nonumber\\
            &+\O\left(|\mu_k|^{\min\{2,p-1\}}\norm*{\eta_k}^2_{H^1(\C)}+\mu_k^4\right).
        \end{align}
        In this computation we also use the H\"older inequality and the following two inequalities
        \begin{equation*}
            \begin{aligned}
                2|\mu_k|^3\norm*{\eta_k}_{H^1(\C)}\leq\; \mu_k^2\norm*{\eta_k}^2_{H^1(\C)}+\mu_k^4,\quad|\mu_k|^{p-1}\norm*{\eta_k}^p_{H^1(\C)}\leq&\;|\mu_k|^{p-1}\norm*{\eta_k}^2_{H^1(\C)}.
            \end{aligned}
        \end{equation*}
        Now it follows from \eqref{tt1}, \eqref{ff1} and Lemma \ref{le2} that the value of
            \begin{align}\label{ff2}
                &\;\mu_k^2\norm*{V_0^\frac{p}{2}\theta_n}^2_{H^1(\C)}+2\mu_k^2\norm*{\eta_k}^2_{H^1(\C)}-\mu_k\int_{\C}|v_k|^{p-2}v_k\left(V_0^\frac{p}{2}\theta_n+2\eta_k\right)\nonumber\\
                =&\;\mu_k^2(p-1)\left(1-\alpha_k^{p-2}\right)\int_\C V_0^{p-2}\left(V_0^\frac{p}{2}\theta_n\right)^2+2\mu_k^2\left(1-C|\mu_k|^{\min\{1,p-2\}}\right)\norm*{\eta_k}_{H^1(\C)}^2\nonumber\\
                &-2\mu_k^2(p-1)\int_\C\left(\alpha_k V_0\right)^{p-2}\eta_k^2-2\mu_k^3(p-1)(p-2)\int_\C\left(\alpha_k V_0\right)^{p-3}\left(V_0^\frac{p}{2}\theta_n\right)^2\eta_k\\
                &-\mu_k^4\frac{(p-1)(p-2)(p-3)}{6}\int_\C\left(\alpha_k V_0\right)^{p-4}\left(V_0^\frac{p}{2}\theta_n\right)^4+\O\left(|\mu_k|^5\right) \nonumber
        \end{align}
        can be bounded above by $C|\mu_k|\norm*{H_1(v_k)}_{H^{-1}(\C)}$.

        To simplify \eqref{ff2}, we need to find appropriate estimates for $\left|\alpha_k^{p-2}-1\right|$ and $\left|\alpha_k-1\right|$. Note that \eqref{qq4} does not meet our requirements. Let us test the identity in \eqref{id} with $V_0$ again and use the orthogonality conditions in \eqref{or3} to derive
        \begin{equation}\label{ff3}
            \begin{aligned}
                \left|\left(\alpha_k^{p-1}-\alpha_k\right)\int_\C V_0^{p}+\int_\C\left(|v_k|^{p-2}v_k-\alpha_k^{p-1}V_0^{p-1}\right)V_0\right|\lesssim \norm*{H_1(v_k)}_{H^{-1}(\C)}
            \end{aligned}
        \end{equation}
        To control the second term, we apply the high order expansion technique and Lemma \ref{le1} once again. In the first region $\left\{|\eta_k|<V_0^{\frac{p}{2}}\right\}$, we have
        \begin{align}\label{ff4}
            &\;\left(|v_k|^{p-2}v_k-\alpha_k^{p-1}V_0^{p-1}\right)V_0\nonumber\\
            =&\;\mu_k(p-1)\alpha_k^{p-2}V_0^{p-1}\left(V_0^\frac{p}{2}\theta_n+\eta_k\right)\nonumber\\
            &+\mu_k^2\frac{(p-1)(p-2)}{2}\alpha_k^{p-3}V_0^{p-2}\left(V_0^\frac{p}{2}\theta_n+\eta_k\right)^2+\O\left(|\mu_k|^3V_0^{\frac{5}{2}p-3}\right)\nonumber\\
            =&\;\mu_k(p-1)\alpha_k^{p-2}V_0^{p-1}\left(V_0^\frac{p}{2}\theta_n+\eta_k\right)+\mu_k^2\frac{(p-1)(p-2)}{2}\alpha_k^{p-3}V_0^{p-2}\left(V_0^\frac{p}{2}\theta_n\right)^2\\
            &+\mu_k^2(p-1)(p-2)\alpha_k^{p-3}V_0^{p-2}\left(V_0^\frac{p}{2}\theta_n\right)\eta_k+\O\left(|\mu_k|^3V_0^{\frac{5}{2}p-3}+\mu_k^2\eta_k^2\right)\nonumber.
        \end{align}
        In the second region $\left\{|\eta_k|\geq V_0^{\frac{p}{2}}\right\}$, we have
        \begin{align}\label{ff5}
                &\;\left(|v_k|^{p-2}v_k-\alpha_k^{p-1}V_0^{p-1}\right)V_0\nonumber\\
                =&\;\mu_k(p-1)\alpha_k^{p-2}V_0^{p-1}\left(V_0^\frac{p}{2}\theta_n+\eta_k\right)+\O\left(\chi(p>3)\mu_k^2\eta_k^2+\mu_k^{p-1}V_0|\eta_k|^{p-1}\right)\nonumber\\
                =&\;\mu_k(p-1)\alpha_k^{p-2}V_0^{p-1}\left(V_0^\frac{p}{2}\theta_n+\eta_k\right)+\mu_k^2\frac{(p-1)(p-2)}{2}\alpha_k^{p-3}V_0^{p-2}\left(V_0^\frac{p}{2}\theta_n\right)^2\nonumber\\
                 &+\mu_k^2(p-1)(p-2)\alpha_k^{p-3}V_0^{p-2}\left(V_0^\frac{p}{2}\theta_n\right)\eta_k+\O\left(\mu_k^2\eta_k^2+\mu_k^{p-1}V_0|\eta_k|^{p-1}\right).
            \end{align}
        The combination of \eqref{ff4}, \eqref{ff5} and the orthogonality conditions in \eqref{or4} yields
        \begin{align}\label{ff6}
            \int_\C\left(|v_k|^{p-2}v_k-\alpha_k^{p-1}V_0^{p-1}\right)V_0=&\;\mu_k^2\frac{(p-1)(p-2)}{2}\alpha_k^{p-3}\int_\C V_0^{p-2}\left(V_0^\frac{p}{2}\theta_n\right)^2\nonumber\\
            &+\O\left(|\mu_k|^3+|\mu_k|^{\min\{2,p-1\}}\norm*{\eta_k}_{H^1(\C)}^{\min\{2,p-1\}}\right).
        \end{align}
It follows from \eqref{ff3} and \eqref{ff6} that
    \begin{align}\label{ff7}
        &\;\left|\left(\alpha_k^{p-1}-\alpha_k\right)\int_\C V_0^{p}+\mu_k^2\frac{(p-1)(p-2)}{2}\alpha_k^{p-3}\int_\C V_0^{p-2}\left(V_0^\frac{p}{2}\theta_n\right)^2\right|\nonumber\\
        \lesssim&\;\norm*{H_1(v_k)}_{H^{-1}(\C)}+\O\left(|\mu_k|^3+|\mu_k|^{\min\{2,p-1\}}\norm*{\eta_k}_{H^1(\C)}^{\min\{2,p-1\}}\right).
    \end{align}
One can easily derive from \eqref{ff7} that
\begin{align}\label{ff8}
        &\;\left|\left(\alpha_k^{p-2}-1\right)+\mu_k^2\frac{(p-1)(p-2)}{2}\alpha_k^{p-4}\frac{\int_\C V_0^{p-2}\left(V_0^\frac{p}{2}\theta_n\right)^2}{\int_\C V_0^{p}}\right|\nonumber\\
        \lesssim&\;\norm*{H_1(v_k)}_{H^{-1}(\C)}+\O\left(|\mu_k|^{\min\{3,2p-2\}}+|\mu_k|^{\min\{2,2p-4\}}\norm*{\eta_k}_{H^1(\C)}^2\right)
    \end{align}
and
\begin{equation}\label{ff9}
    \begin{aligned}
        |\alpha_k-1|\lesssim \norm*{H_1(v_k)}_{H^{-1}(\C)}+\O\left(\mu_k^2+|\mu_k|^{\min\{2,2p-4\}}\norm*{\eta_k}_{H^1(\C)}^2\right).
    \end{aligned}
\end{equation}
Combining \eqref{ff8} and \eqref{ff9} into \eqref{ff2} yields that the value of
\begin{equation}\label{ff10}
    \begin{aligned}
        &\;2\mu_k^2\left(1-C|\mu_k|^{\min\{1,p-2\}}\right)\norm*{\eta_k}_{H^1(\C)}^2-2\mu_k^2(p-1)\int_\C V_0^{p-2}\eta_k^2\\
        &-2\mu_k^3(p-1)(p-2)\int_\C V_0^{p-3}\left(V_0^\frac{p}{2}\theta_n\right)^2\eta_k+\mu_k^4\frac{(p-1)^2(p-2)}{2}\frac{\left(\int_\C V_0^{p-2}\left(V_0^\frac{p}{2}\theta_n\right)^2\right)^2}{\int_\C V_0^{p}}\\
        &-\mu_k^4\frac{(p-1)(p-2)(p-3)}{6}\int_\C V_0^{p-4}\left(V_0^\frac{p}{2}\theta_n\right)^4+\O\left(|\mu_k|^5\right)
    \end{aligned}
\end{equation}
can be bounded above by $C|\mu_k|\norm*{H_1(v_k)}_{H^{-1}(\C)}$. From Lemma \ref{le4} we know that the value of \eqref{ff10} is larger than $2\mu_k^4\left(E_0+F+\O\left(|\mu_k|^{\min\{1,p-2\}}\right)\right)$. Thus
\begin{equation*}
    \liminf_{k\rightarrow +\infty}|\mu_k|^{-3}\norm*{H_1(v_k)}_{H^{-1}(\C)}\gtrsim E_0+F>0.
\end{equation*}
This contradicts \eqref{qq6} and \eqref{qq5}.

Now we show that the cubic estimate \eqref{main-re4} is optimal. Choose a sequence $\{\mu_k\}_k$ of numbers such that $\mu_k\neq 0$ and $\mu_k\rightarrow 0$. Set
\begin{equation}\label{cons}
    w_k:=\left(1-C_0\mu_k^2\right)V_0+\mu_k\left(V_0^{\frac{p}{2}}\theta_n+\mu_k\eta\right).
\end{equation}
The constant $C_0=C_0(p,n)$ and the function $\eta\in H^1(\C)$ will be determined later. We require $\eta$ to satisfy the orthogonality conditions in \eqref{or3} and \eqref{or4}. Moreover, we need $|\eta|\lesssim V_0$. Thanks to Lemma \ref{le1} and our assumptions, we have the expansion
\begin{align}
    &\;\partial_s^2 w_k+\Delta_\theta w_k-\Lambda w_k+|w_k|^{p-2}w_k\nonumber\\
    =&-\left(1-C_0\mu_k^2\right)V_0^{p-1}-\mu_k(p-1)V_0^{p-2}V_0^\frac{p}{2}\theta_n+\mu_k^2\left(\partial_s^2 \eta+\Delta_\theta \eta-\Lambda \eta\right)\nonumber\\
    &+\left(1-C_0\mu_k^2\right)^{p-1}V_0^{p-1}+\mu_k(p-1)\left(1-C_0\mu_k^2\right)^{p-2}V_0^{p-2}V_0^\frac{p}{2}\theta_n\nonumber\\
    &+\mu_k^2(p-1)\left(1-C_0\mu_k^2\right)^{p-2}V_0^{p-2}\eta+\mu_k^2\frac{(p-1)(p-2)}{2}\left(1-C_0\mu_k^2\right)^{p-3}V_0^{2p-3}\theta_n^2\nonumber\\
    &+\O\left(|\mu_k|^3V_0^{p-1}\right)\nonumber\\
    =&-(p-2)C_0\mu_k^2V_0^{p-1}+\mu_k^2\left(\partial_s^2 \eta+\Delta_\theta \eta-\Lambda \eta\right)+\mu_k^2(p-1)V_0^{p-2}\eta\nonumber\\
    &+\mu_k^2\frac{(p-1)(p-2)}{2}V_0^{2p-3}\theta_n^2+\O\left(|\mu_k|^3V_0^{p-1}\right).\nonumber
\end{align}
Note that
\begin{equation*}
    \inf_{t\in\R}\norm*{w_k- V_t}_{H^1(\C)}=|\mu_k|\norm*{V_0^{\frac{p}{2}}\theta_n}_{H^1(\C)}+\O\left(\mu_k^2\right).
\end{equation*}
It suffices to choose suitable $C_0$ and $\eta$ such that
\begin{equation}\label{cc1}
    -\partial_s^2 \eta-\Delta_\theta \eta+\Lambda \eta-(p-1)V_0^{p-2}\eta+(p-2)C_0V_0^{p-1}=\frac{(p-1)(p-2)}{2}V_0^{2p-3}\theta_n^2.
\end{equation}
Assume $\eta$ has the following expansion in spherical harmonics
$$\eta=\eta_1\left(\theta_n^2-\frac{1}{n}\right)+\eta_2.$$
Then \eqref{cc1} is equivalent to
\begin{equation*}
\begin{aligned}
    \frac{(p-1)(p-2)}{2}V_0^{2p-3}\theta_n^2=&\;\left(-\partial_s^2 \eta_1+(2n+\Lambda) \eta_1-(p-1)V_0^{p-2}\eta_1\right)\left(\theta_n^2-\frac{1}{n}\right)\\
    &-\partial_s^2 \eta_2+\Lambda \eta_2-(p-1)V_0^{p-2}\eta_2+(p-2)C_0V_0^{p-1}.
\end{aligned}
\end{equation*}
Thanks to Lemma \ref{le3}, we can take $\eta_1$ to be the unique solution of
\begin{equation*}
    -\partial_s^2 \eta_1+(2n+\Lambda) \eta_1-(p-1)V_0^{p-2}\eta_1=\frac{(p-1)(p-2)}{2}V_0^{2p-3}.
\end{equation*}
It remains to choose $C_0$ and $\eta_2$ such that
\begin{equation*}
    -\partial_s^2 \eta_2+\Lambda \eta_2-(p-1)V_0^{p-2}\eta_2+(p-2)C_0V_0^{p-1}=\frac{(p-1)(p-2)}{2n}V_0^{2p-3}.
\end{equation*}
In fact, by direct computation we have
\begin{equation*}
    -\partial_s^2 V_0^{p-1}+\Lambda V_0^{p-1}-(p-1)V_0^{p-2}V_0^{p-1}=\frac{2(p-1)(p-2)}{p}V_0^{2p-3}-p(p-2)\Lambda V_0^{p-1}
\end{equation*}
and
\begin{equation*}
    -\partial_s^2 V_0+\Lambda V_0-(p-1)V_0^{p-2}V_0=-(p-2)V_0^{p-1}.
\end{equation*}
Therefore, to achieve \eqref{cc1}, we only need to take
\begin{equation*}
    \eta_2:=\frac{p}{4n}V_0^{p-1}-\frac{p\int_\C V_0^{2p-2}}{4n\int_\C V_0^p}V_0
\end{equation*}
and
\begin{equation*}
    C_0:=\frac{p^2}{4n}\Lambda-\frac{p\int_\C V_0^{2p-2}}{4n\int_\C V_0^p}.
\end{equation*}

Next we derive \eqref{main-re5}. Given $v\in H^1(\C)$. Assume $V_0$ attains the infimum in \eqref{main-re4} and set $\rho:=v-V_0$. Define $c_0=\left(\int_\C V_0^p\right)^{-1}\int_\C \rho V_0^{p-1}$, $\rho_1=\Pi_{Y_0}(\rho)$ and $\rho_2=\Pi_{Y_0}^\perp (\rho)-c_0V_0$. We have $\rho=c_0V_0+\rho_1+\rho_2$. By testing $H_1(v)$ against $V_0$ and using Lemma \ref{le1}, we can obtain 
\begin{align}\label{edc1}
    |c_0|\lesssim \norm*{H_1(v)}_{H^{-1}(\C)}+\norm*{\rho_1}^{\min\{2,p-1\}}_{H^1(\C)}+\norm*{\rho_2}^{\min\{2,p-1\}}_{H^1(\C)}.
\end{align}
By testing $H_1(v)$ against $\rho_2$ and using orthogonality conditions, we deduce
\begin{align}
    \norm*{\rho_2}_{H^1(\C)}^2 \leq&\; \norm*{H_1(v)}_{H^{-1}(\C)}\norm*{\rho_2}_{H^1(\C)}+(p-1)(1+c_0)^{p-2}\int_\C V_0^{p-2}\rho_2^2\nonumber\\
    &+C\norm*{\rho_1}^{\min\{2,p-1\}}_{H^1(\C)}\norm*{\rho_2}_{H^1(\C)}+C\norm*{\rho_2}^{\min\{3,p\}}_{H^1(\C)}.\nonumber
\end{align}
It follows from Lemma \ref{le2} that
\begin{align}
    \norm*{\rho_2}_{H^1(\C)}^2\geq \gamma_3 \int_\C V_0^{p-2}\rho_2^2\nonumber
\end{align}
and $\gamma_3>p-1$. Hence 
\begin{align}\label{edc2}
    \norm*{\rho_2}_{H^1(\C)} \lesssim \norm*{H_1(v)}_{H^{-1}(\C)}+\norm*{\rho_1}^{\min\{2,p-1\}}_{H^1(\C)}.
\end{align}
From the assumptions we have
\begin{align}\label{edc3}
    \norm*{\rho_1}^{\min\{2,p-1\}}_{H^1(\C)}\ll |c_0|+\norm*{\rho_2}_{H^1(\C)}.
\end{align}
Combining \eqref{edc1}, \eqref{edc2} and \eqref{edc3} yields
\begin{align}
    |c_0|+\norm*{\rho_2}_{H^1(\C)} \lesssim \norm*{H_1(v)}_{H^{-1}(\C)},
\end{align}
which is equivalent to \eqref{main-re5}.
\end{proof}

\begin{proof}[Proof of Theorem \ref{thm1}]
    From the global compactness principle (see \cite[Proposition 3.2]{Wei} for example), we know that for any $\varepsilon>0$, if
    \begin{equation*}
        \norm*{H(u)}_{D_a^{-1,2}(\R^n)}\leq\varepsilon,
    \end{equation*}
    then $u$ is $o_\varepsilon(1)$-close to $\sum\limits_{i=1}^\nu W_i$. Here $\nu\in\N_+$ and $W_i\;(1\leq i\leq\nu)$ are nontrivial solutions to the equation \eqref{eq}. Moreover, it holds that
    \begin{equation*}
        \norm*{u}_{D_a^{1,2}(\R^n)}^2=\sum_{i=1}^\nu \norm*{W_i}_{D_a^{1,2}(\R^n)}^2+o_\varepsilon(1).
    \end{equation*}
    Note that for any sign-changing solution $W$ to the equation \eqref{eq}, we have the energy estimate
    \begin{align}
        \norm*{W}^2\geq 2S(a,b,n)^\frac{p}{p-2}.\nonumber
    \end{align}
    It follows from \eqref{zzz1} that when $\varepsilon$ is small, it must hold that $\nu=1$ and $W_1$ is a Talenti bubble. Now we can take a number $\varepsilon_0=\varepsilon_0(p,n)$ such that if
    \begin{align}\label{aa1}
        \norm*{H(u)}_{D_a^{-1,2}(\R^n)}\leq\varepsilon_0,
    \end{align}
    then 
    \begin{equation*}
        \inf_{\lambda>0}\norm*{u-U_\lambda}_{D_a^{1,2}(\R^n)}\leq\delta.
    \end{equation*}
    Here $\delta$ is the constant that appears in Theorem \ref{thm3} and Theorem \ref{thm5}. It follows from Theorem \ref{thm3} that \eqref{main-re1} and \eqref{main-re2} hold when \eqref{aa1} holds. If \eqref{aa1} does not hold, \eqref{main-re1} and \eqref{main-re2} clearly hold since the left sides of them have trivial upper bounds. One can verify the optimality of \eqref{main-re1} by choosing the functions $w_k$ defined in \eqref{cons}.
\end{proof}
\begin{remark}\label{re1}
    In the \emph{Case 2} in our proof of Theorem \ref{thm5}, we choose $\mu_k\left(V_0^\frac{p}{2}\theta_n+2\eta_k\right)$ as the test function. A more natural choice, however, is $\mu_k\left(V_0^\frac{p}{2}\theta_n+\eta_k\right)$, as it describes the distance from the function $v_k$ to the manifold $\left\{\lambda V_t\;\big|\;\lambda,t\in\R\right\}$. However, this choice hinders the proof over the Felli-Schneider curve's full parameter range. Indeed, for any $\lambda>0$, if we use $\mu_k\left(V_0^\frac{p}{2}\theta_n+\lambda\eta_k\right)$ as the test function, we will finally deduce
    \begin{equation*}
    \liminf_{k\rightarrow +\infty}|\mu_k|^{-3}\norm*{H_1(v_k)}_{H^{-1}(\C)}\gtrsim \frac{(\lambda+2)^2}{4\lambda}E_0+2F.
\end{equation*}
If $\lambda\neq 2$, from Lemma \ref{le4} we know that our proof fails when $n\geq3$ and $p$ is close to $2^*$. Therefore, our initial choice of the test function appears to be optimal.
\end{remark}
\vspace{10pt}
\begin{remark}\label{re2}
    The construction \eqref{cons} of the sequence $\{w_k\}_k$ in the proof of Theorem \ref{thm5} is technical. Since the degeneracy of the equation \eqref{eq1} stems from the nontrivial functions $V_0^{\frac{p}{2}}\theta_i,1\leq i\leq n$, the sequence $\left\{\hat{w}_k:=V_0+\mu_kV_0^\frac{p}{2}\theta_n\right\}_k$ appears to be a natural choice. However, this sequence can only yield a quadratic estimate. The sequence $\{w_k\}_k$ can be viewed as an appropriate perturbation of $\{\hat{w}_k\}_k$. The process of constructing the sequence $\{w_k\}_k$ demonstrates that our choices of the number $C_0$ and the function $\eta$ are nearly optimal for obtaining a cubic estimate.
\end{remark}
\vspace{10pt}
\begin{remark}\label{re3}
    Here we derive the asymptotic estimate \eqref{asy}. Let us first give a reformulation of it: let $\{v_k\}_k$ be a sequence in $H^1(\C)$ such that $v_k\notin \mathcal{M}_{p,\Lambda,n}$ and $$\inf_{t\in\R}\norm*{v_k- V_t}_{H^1(\C)}\rightarrow 0\;\;\text{as }k\rightarrow+\infty,$$ then there exists a positive number $R(p,n)$ such that
    \begin{equation}\label{asy1}
        \liminf_{k\rightarrow+\infty}\frac{\norm*{H_1(v_k)}_{H^{-1}(\C)}}{\inf\limits_{t\in\R}\norm*{v_k- V_t}_{H^1(\C)}^3}\geq R(p,n).
    \end{equation}
    Without loss of generality, we assume that the limit of the left term exists and is finite. Reviewing the proof of Theorem \ref{thm5} with more attention to the constant $C$, we can show that the \emph{Case 2} must hold and the value of \eqref{ff10} can be bounded above by
    \begin{equation*}
        \left(|\mu_k|\norm*{V_0^{\frac{p}{2}}\theta_n}_{H^1(\C)}+\O\left(\mu_k^2\right)\right)\norm*{H_1(v_k)}_{H^{-1}(\C)}.
    \end{equation*}
    Thanks to Lemma \ref{le4}, we know that \eqref{ff10} has the lower bound
    \begin{equation*}
        2\mu_k^4\left(E_0+F+\O\left(|\mu_k|^{\min\{1,p-2\}}\right)\right).
    \end{equation*}
    Thus we have
    \begin{equation*}
        \liminf_{k\rightarrow+\infty}|\mu_k|^{-3}\norm*{H_1(v_k)}_{H^{-1}(\C)}\geq 2(E_0+F)\norm*{V_0^{\frac{p}{2}}\theta_n}_{H^1(\C)}^{-1}.
    \end{equation*}
    From \eqref{ff7} we get
    \begin{equation*}
        |\alpha_k-1|=\O(|\mu_k|^{\min\{2,p-1\}}).
    \end{equation*}
    Therefore,
    \begin{equation*}
        \inf_{t\in\R}\norm*{v_k- V_t}_{H^1(\C)}\leq \norm*{v_k- V_0}_{H^1(\C)}=|\mu_k|\left(\norm*{V_0^{\frac{p}{2}}\theta_n}_{H^1(\C)}+\norm*{\eta_k}_{H^1(\C)}+\O(|\mu_k|^{\min\{1,p-2\}})\right).
    \end{equation*}
    Combining all the arguments above, we can obtain
    \begin{equation*}
        \liminf_{k\rightarrow+\infty}\frac{\norm*{H_1(v_k)}_{H^{-1}(\C)}}{\inf\limits_{t\in\R}\norm*{v_k- V_t}_{H^1(\C)}^3}\geq 2(E_0+F)\norm*{V_0^{\frac{p}{2}}\theta_n}_{H^1(\C)}^{-4}=:R(p,n).
    \end{equation*}
    Since the value of $E_0+F$ has been calculated in \cite[Section 4]{Fra1}, here we can compute
    \begin{align}\label{rrr}
        R(p,n)=&\;\alpha\beta^{-p}\pi^{-\frac{1}{2}}|\S^{n-1}|^{-1}\frac{2p(p-2)}{5p-6}\frac{\Gamma\left(\frac{2p-2}{p-2}+\frac{1}{2}\right)}{\Gamma\left(\frac{2p-2}{p-2}\right)}\nonumber\\
        &\times \left(\frac{3p-4}{4p-4}-\frac{pn-3n}{pn+2p}-\frac{n-1}{n+2}\sum_{k=0}^{+\infty}\frac{P(k-\xi)-P(k)}{P(-1)}\right),
    \end{align}
    where
    \begin{equation*}
        P(x):=\frac{\Gamma\left(x+\frac{3}{2}\right)\Gamma\left(x+2\xi_1-1\right)\Gamma\left(x+2\xi_1\right)}{\Gamma\left(x+\xi_1-\xi_2+1\right)\Gamma\left(x+\xi_1+\xi_2+1\right)\Gamma\left(x+2\xi_1+\frac{1}{2}\right)},
    \end{equation*}
    $\alpha:=\frac{p-2}{2}\sqrt{\Lambda}$, $\beta:=\left(\frac{p\Lambda}{2}\right)^{\frac{1}{p-2}}$, $\xi_1:=\frac{2p-3}{p-2}$, $\xi_2:=\frac{\sqrt{1+\frac{2n}{\Lambda}}}{p-2}$ and $\xi:=\xi_1-\xi_2$.
\end{remark}

\section{Proofs of Theorem \ref{thm2} and Theorem \ref{thm6}}\label{sec4}
In this section we first establish Theorem \ref{thm6}. Theorem \ref{thm4} can then be obtained by using the Emden-Fowler transformation. Combining Theorem \ref{thm4} and the global compactness principle yields Theorem \ref{thm2}.  
\begin{proof}[Proof of Theorem \ref{thm6}]
    Assume 
    \begin{align}
        \inf_{\substack{s_i\in\R\\
    1\leq i\leq\nu}}\norm*{v-\sum_{i=1}^\nu V_{s_i}}_{H^1(\C)}=\norm*{v-\sum_{i=1}^{\nu}V_{t_i}}_{H^1(\C)}\leq \delta\nonumber
    \end{align}
    and $|t_i-t_j|\geq \delta^{-1}$ for any $i\neq j$. For simplicity, we set
    \begin{align}
        \sigma:=\sum_{i=1}^{\nu}V_{t_i},\quad \rho:=v-\sigma,\quad R:=\min_{i\neq j}\{|t_i-t_j|\},\quad Q:=e^{-\sqrt{\Lambda}R}.\nonumber
    \end{align}
    We also assume $-\infty=t_0<t_1<t_2<\cdots<t_\nu<t_{\nu+1}=+\infty$ and $V_{t_0}=V_{t_{\nu+1}}=0$. By variation we can obtain the following orthogonality conditions:
    \begin{equation}\label{or5}
        \left\langle \rho,\partial_s V_{t_i}\right\rangle_{H^1(\C)}=\left\langle \rho,V_{t_i}^{p-2}\partial_s V_{t_i}\right\rangle_{L^2(\C)}=0\quad \text{for }1\leq i\leq\nu.
    \end{equation}
    
    Let us first provide a lower bound for $\norm*{H_1(v)}_{H^{-1}(\C)}$. Specifically, we prove that, when $\delta$ is small, it holds that
    \begin{equation}\label{lower bd}
        \norm*{H_1(v)}_{H^{-1}(\C)}+\norm*{\rho}_{H^1(\C)}^2+\max_{1\leq k\leq\nu}\left|\int_\C\sigma^{p-2}\rho\;\partial_s V_{t_k}\right|\gtrsim Q.
    \end{equation}
    We adapt the ideas in \cite{Den,Wei} to derive \eqref{lower bd}. Consider the following identity:
    \begin{equation*}
    \begin{aligned}
        \partial_s^2 v+\Delta_\theta v-\Lambda v+|v|^{p-2}v=&\;\partial_s^2 \rho+\Delta_\theta \rho-\Lambda \rho+\left(|v|^{p-2}v-\sigma^{p-1}-(p-1)\sigma^{p-2}\rho\right)\\
        &+\left(\sigma^{p-1}-\sum_{i=1}^\nu V_{t_i}^{p-1}\right)+(p-1)\sigma^{p-2}\rho.
    \end{aligned}
    \end{equation*}
    Fix a number $1\leq k\leq \nu$. Testing this identity with $\partial_s V_{t_k}$ and using the orthogonality conditions in \eqref{or5}, we deduce
    \begin{equation}\label{est1}
    \begin{aligned}
        &\;\norm*{H_1(v)}_{H^{-1}(\C)}+\left|\int_\C\sigma^{p-2}\rho\;\partial_s V_{t_k}\right|\\
        \gtrsim&\;\int_{\C}\left(|v|^{p-2}v-\sigma^{p-1}-(p-1)\sigma^{p-2}\rho\right)\partial_s V_{t_k}+\int_{\C}\left(\sigma^{p-1}-\sum_{i=1}^\nu V_{t_i}^{p-1}\right)\partial_s V_{t_k}.
    \end{aligned}
    \end{equation}
    We need to estimate the two integrals on the right hand side. As for the first integral, we split $\C$ into two regions: $\{2|\rho|\leq \sigma\}$ and $\{2|\rho|>\sigma\}$. In the first region, based on Lemma \ref{le1} and the fact that $|\partial_s V_{t_k}|\lesssim V_{t_k}$, we have
    \begin{equation*}
        \left|\left(|v|^{p-2}v-\sigma^{p-1}-(p-1)\sigma^{p-2}\rho\right)\partial_s V_{t_k}\right|\lesssim C\sigma^{p-2}\rho^2.
    \end{equation*}
    In the second region, we can directly derive
    \begin{equation*}
        \left|\left(|v|^{p-2}v-\sigma^{p-1}-(p-1)\sigma^{p-2}\rho\right)\partial_s V_{t_k}\right|\lesssim |\rho|^{p}.
    \end{equation*}
    Thus we can obtain
    \begin{equation}\label{est2}
        \left|\int_{\C}\left(|v|^{p-2}v-\sigma^{p-1}-(p-1)\sigma^{p-2}\rho\right)\partial_s V_{t_k}\right|\lesssim \norm*{\rho}_{H^1(\C)}^2.
    \end{equation}
    As for the second integral, we decompose it into four parts:
    \begin{align}
        \int_{\C}\left(\sigma^{p-1}-\sum_{i=1}^\nu V_{t_i}^{p-1}\right)\partial_s V_{t_k}=\I_1+\I_2+\I_3+\I_4,\nonumber
    \end{align}
    where
    \begin{align}
        &\I_1=\int_{\{\sigma_k\leq V_{t_k}\}}\left((p-1)V_{t_k}^{p-2}\sigma_k-\sum_{i=1,i\neq k}^{\nu}V_{t_i}^{p-1}\right)\partial_s V_{t_k},\nonumber\\
        &\I_2=\int_{\{\sigma_k\leq V_{t_k}\}}\left(\sigma^{p-1}-V_{t_k}^{p-1}-(p-1)V_{t_k}^{p-2}\sigma_k\right)\partial_s V_{t_k},\nonumber\\
        &\I_3=\int_{\{\sigma_k> V_{t_k}\}}\left((p-1)\sigma_k^{p-2}V_{t_k}+\sigma_k^{p-1}-\sum_{i=1}^{\nu}V_{t_i}^{p-1}\right)\partial_s V_{t_k},\nonumber\\
        &\I_4=\int_{\{\sigma_k> V_{t_k}\}}\left(\sigma^{p-1}-\sigma_k^{p-1}-(p-1)\sigma_k^{p-2}V_{t_k}\right)\partial_s V_{t_k},\nonumber
    \end{align}
    and $\sigma_k:=\sigma-V_{t_k}$. From Lemma \ref{le1} and Lemma \ref{le5} we have
    \begin{equation}\label{I1}
    \begin{aligned}
        &|\I_2|\lesssim \int_{\{\sigma_k\leq V_{t_k}\}}V_{t_k}^{p-2}\sigma_k^2\lesssim \int_{\C}V_{t_k}^{p-1-\varepsilon}\sigma_k^{1+\varepsilon}\lesssim Q^{1+\varepsilon},\\
        &|\I_4|\lesssim \int_{\{\sigma_k> V_{t_k}\}}\sigma_k^{p-2}V_{t_k}^2\lesssim \int_{\C}\sigma_k^{p-1-\varepsilon}V_{t_k}^{1+\varepsilon}\lesssim Q^{1+\varepsilon}.
    \end{aligned}
    \end{equation}
    Here $\varepsilon=\varepsilon(p)$ is a small positive number such that $\varepsilon<\frac{p-2}{3}$. As for $\I_3$, thanks to Lemma \ref{le5}, we have
    \begin{align}\label{I2}
        |\I_3|\lesssim&\; \int_{\{\sigma_k> V_{t_k}\}}\left(\sigma_k^{p-2}V_{t_k}^2+V_{t_k}^p\right)+\int_{\{\sigma_k> V_{t_k}\}}\left|\sigma_k^{p-1}-\sum_{i=1,i\neq k}^{\nu}V_{t_i}^{p-1}\right|V_{t_k}\nonumber\\
        \lesssim&\; \int_{\C}\sigma_k^{p-1-\varepsilon}V_{t_k}^{1+\varepsilon}+\sum_{\substack{1\leq i< j\leq \nu\\
        i,j\neq k}}\int_{\C}V_{t_i}^{p-2}V_{t_j}V_{t_k}\nonumber\\
        \lesssim&\; Q^{1+\varepsilon}.
    \end{align}
    In the last step of the above estimate, we use the H\"older inequality and the observation that for any $i,j\neq k$, 
    \begin{equation*}
        \max\{|t_i-t_j|,|t_i-t_k|,|t_j-t_k|\}\geq 2R.
    \end{equation*}
    It remains to consider $\I_1$. Note that
    \begin{align}\label{I3}
        \I_1=&\;(p-1)\left(\int_\C-\int_{\{\sigma_k> V_{t_k}\}}\right) V_{t_k}^{p-2}\partial_s V_{t_k}\sigma_k-\sum_{i=1,i\neq k}^{\nu}\int_{\{\sigma_k\leq V_{t_k}\}}V_{t_i}^{p-1}\partial_s V_{t_k}\nonumber\\
        \geq&\;\sum_{i\neq k}\int_\C\partial_s (V_{t_k})^{p-1}V_{t_i}-C\int_{\{\sigma_k>V_{t_k}\}}V_{t_k}^{p-1-\varepsilon}\sigma_k^{1+\varepsilon}-C\int_{\C}\sigma_k^{p-1-\varepsilon}V_{t_k}^{1+\varepsilon}\\
        \geq&\;\sum_{i\neq k}\int_\C \partial_s V_{t_k} V_{t_i}^{p-1}-CQ^{1+\varepsilon}.\nonumber
    \end{align}
    From the expression of $V_{t}(s)$ (see \eqref{sol1}), we know that 
    \begin{align}
        \int_{\C} \partial_sV_{t_i} V_{t_j}^{p-1}\geq 0\quad\text{if and only if}\;\;i\geq j.\nonumber
    \end{align}
    For any $j\geq k+2$, from Lemma \ref{le5} and the estimate $|\partial_s V_{t_k}|\lesssim V_{t_k}$, we can deduce
    \begin{align}
        -\int_\C \partial_s V_{t_k} V_{t_j}^{p-1}\lesssim Q^2.\nonumber
    \end{align}
    Based on Lemma \ref{le6} and the above arguments, we can find three positive constants $C_1,C_2,C_3$ such that
    \begin{equation*}
        \sum_{i\neq k}\int_\C \partial_s V_{t_k} V_{t_i}^{p-1}\geq C_1e^{-\sqrt{\Lambda}|t_{k-1}-t_{k}|}-C_2e^{-\sqrt{\Lambda}|t_k-t_{k+1}|}-C_3Q^2.
    \end{equation*}
    Now set
    \begin{equation*}
        k_0=\sup\{1\leq i\leq \nu\;|\;|t_i-t_{i-1}|=R\}
    \end{equation*}
    and
    \begin{equation*}
        k_1=\sup\left\{k_0\leq i< \nu\;\big|\;t_{j+1}-t_j<t_j-t_{j-1}+\Lambda^{-\frac{1}{2}}\ln (2C_2C_1^{-1})\;\text{for any }k_0\leq j\leq i\right\}.
    \end{equation*}
    If $k_0=\nu$, choose $k=k_0$. If not, choose $k=k_1$. It is not hard to check that
    \begin{equation}\label{I4}
        \sum_{i\neq k}\int_\C \partial_s V_{t_k} V_{t_i}^{p-1}\gtrsim Q
    \end{equation}
    provided $\delta$ is sufficiently small. Combining the estimates \eqref{est1}, \eqref{est2}, \eqref{I1}, \eqref{I2}, \eqref{I3} and \eqref{I4} gives \eqref{lower bd}.

    Next we need to find an upper bound of $\norm*{\rho}_{H^1(\C)}$. Following similar ideas as in \cite{Den,Wei}, we define
    \begin{equation*}
        Q_{i,j}:=e^{-\sqrt{\Lambda}|t_i-t_j|},\quad \varphi_{t_i}(s):=e^{-\sqrt{\Lambda}|s-t_i|}\quad\text{for }1\leq i,j\leq \nu,
    \end{equation*}
    \begin{align}
        W_1(s):=&\;\sum_{i=1}^{\nu-1}\left(Q_{i,i+1}\varphi_{t_i}^{p-3}(s)+Q\varphi_{t_i}^{1-\zeta}(s)\right)\chi\left(t_i+1\leq s\leq\frac{t_i+t_{i+1}}{2}\right)\nonumber\\
        &+\sum_{i=1}^{\nu-1}\left(Q_{i,i+1}\varphi_{t_{i+1}}^{p-3}(s)+Q\varphi_{t_{i+1}}^{1-\zeta}(s)\right)\chi\left(\frac{t_i+t_{i+1}}{2}\leq s\leq t_{i+1}-1\right)\nonumber\\
        &+Q\varphi_{t_\nu}^{1-\zeta}(s)\chi(t_\nu+1\leq s)+Q\varphi_{t_1}^{1-\zeta}(s)\chi(s\leq t_1-1)\nonumber\\
        &+\sum_{i=1}^{\nu}Q\chi(t_i-1\leq s\leq t_i+1),\nonumber\\
        W_2(s):=&\sum_{i=1}^\nu Q\varphi_{t_i}^{1-\zeta}(s)\chi\left(\frac{t_i+t_{i-1}}{2}\leq s\leq\frac{t_i+t_{i+1}}{2}\right),\nonumber\\
        W_3(s):=&\;\sum_{i=1}^{\nu-1}\left(Q_{i,i+1}\varphi_{t_i}^{p-3}(s)+Q\varphi_{t_i}^{1-\zeta}(s)\right)\chi\left(t_i+2\leq s\leq\frac{t_i+t_{i+1}}{2}\right)\nonumber\\
        &+\sum_{i=1}^{\nu-1}\left(Q_{i,i+1}\varphi_{t_{i+1}}^{p-3}(s)+Q\varphi_{t_{i+1}}^{1-\zeta}(s)\right)\chi\left(\frac{t_i+t_{i+1}}{2}\leq s\leq t_{i+1}-2\right)\nonumber\\
        &+Q\varphi_{t_\nu}^{1-\zeta}(s)\chi(t_\nu+2\leq s)+Q\varphi_{t_1}^{1-\zeta}(s)\chi(s\leq t_1-2)\nonumber\\
        &+\sum_{i=1}^{\nu}Q\chi(t_i-2\leq s\leq t_i+2).\nonumber
    \end{align}
    Here $\zeta=\zeta(p,n,\nu)>0$ is a sufficiently small number. We also define three new norms:
    \begin{equation*}
        \norm*{h}_{i}:=\sup_{\C}|h(s,\theta)|W_i^{-1}(s),\quad 1\leq i\leq 3.
    \end{equation*}
    Using the expression (see \eqref{sol1}) of Talenti bubbles, we can compute that
    \begin{equation}\label{fin1}
        \norm*{\sigma^{p-1}-\sum_{i=1}^\nu V_{t_i}^{p-1}}_1\lesssim 1
    \end{equation}
    when $2<p<4$ and 
    \begin{equation}\label{fin2}
        \norm*{\sigma^{p-1}-\sum_{i=1}^\nu V_{t_i}^{p-1}}_2\lesssim 1
    \end{equation}
    when $p\geq 4$. Specifically, we have
    \begin{align}
        \left|\sigma^{p-1}-\sum_{i=1}^\nu V_{t_i}^{p-1}\right|\lesssim \left|V_{t_{i+1}}^{p-2}V_{t_i}+V_{t_{i-1}}^{p-2}V_{t_i}\right|\lesssim Q_{i,i+1}\varphi_{t_i}^{p-3}(s)+Q\varphi_{t_i}^{1-\zeta}(s)\nonumber
    \end{align}
    when $t_i+1\leq s\leq\frac{t_i+t_{i+1}}{2}$ and $2<p<4$. The remaining cases can be similarly argued.
    
    By slightly modifying the blow-up arguments used in \cite[Lemma 5.4]{Wei} (see also \cite[Lemma 3.5]{Den}), we can obtain the following a priori estimate: assume $\phi$ solves the equation
    \begin{align}
        -\partial_s^2 \phi-\Delta_\theta \phi+\Lambda \phi-p\sigma^{p-2}\phi =h\nonumber
    \end{align}
    and 
    \begin{align}\label{mm0}
        \left\langle \phi,\partial_s V_{t_i}\right\rangle_{H^1(\C)}=\left\langle \phi,V_{t_i}^\frac{p}{2}\theta_j\right\rangle_{H^1(\C)}=0\quad\text{for any }1\leq i\leq \nu,\;1\leq j\leq n,
    \end{align}
    then 
    \begin{align}\label{mm1}
        \norm*{\phi}_3\chi(p<4)+\norm*{\phi}_2\chi(p\geq 4)\lesssim \norm*{h}_1\chi(p<4)+\norm*{h}_2\chi(p\geq 4).
    \end{align}
    From the standard Fredholm alternative, for any $h$ such that the right hand side of \eqref{mm1} is finite, we can find a function $\phi$ and numbers $d_i,d_{ij}$ such that
    \begin{align}\label{mm2}
        -\partial_s^2 \phi-\Delta_\theta \phi+\Lambda \phi-p\sigma^{p-2}\phi =h+\sum_{i=1}^\nu d_iV_{t_i}^{p-2}\partial_sV_{t_i}+\sum_{i=1}^\nu\sum_{j=1}^n d_{ij}V_{t_i}^{p-2}V_{t_i}^{\frac{p}{2}}\theta_j
    \end{align}
    and $\phi$ satisfies \eqref{mm0}. Note that \eqref{mm1} yields
    \begin{align}\label{mm3}
        \norm*{\phi}_3\chi(p<4)+\norm*{\phi}_2\chi(p\geq 4)\lesssim \norm*{h}_1\chi(p<4)+\norm*{h}_2\chi(p\geq 4)+Q^{-1}\max\{|d_i|,|d_{ij}|\}.
    \end{align}
    By testing \eqref{mm2} with $\partial_s V_{t_i}$ and employing Lemma \ref{le1}, Lemma \ref{le5}, we deduce that
    \begin{align}\label{mm4}
        |d_i|\lesssim&\; o_\delta(1)\sum_{j\neq i}|d_j|+\norm*{h}_1\chi(p<4)\int_{\C}W_1V_{t_i}+\norm*{h}_2\chi(p\geq 4)\int_{\C}W_2V_{t_i}+\left|\int_{\C}\sigma^{p-2}\phi\partial_s V_{t_i}\right|\nonumber\\
        \lesssim&\;o_\delta(1)\sum_{j\neq i}|d_j|+Q\norm*{h}_1\chi(p<4)+Q\norm*{h}_2\chi(p\geq 4)+\left|\int_{\C}(\sigma^{p-2}-V_{t_i}^{p-2})\phi\partial_s V_{t_i}\right|\nonumber\\
        \lesssim&\;o_\delta(1)\sum_{j\neq i}|d_j|+Q\norm*{h}_1\chi(p<4)+Q\norm*{h}_2\chi(p\geq 4)\nonumber\\
        &+\norm*{\phi}_3\chi(p<4)\int_{\C}(\sigma^{p-2}-V_{t_i}^{p-2})W_3V_{t_i}+\norm*{\phi}_2\chi(p\geq 4)\int_{\C}(\sigma^{p-2}-V_{t_i}^{p-2})W_2V_{t_i}\nonumber\\
        \lesssim&\;o_\delta(1)\sum_{j\neq i}|d_j|+Q\norm*{h}_1\chi(p<4)+Q\norm*{h}_2\chi(p\geq 4)\nonumber\\
        &+o_\delta(1)Q\left(\norm*{\phi}_3\chi(p<4)+\norm*{\phi}_2\chi(p\geq 4)\right).
    \end{align}
    Similarly, by testing \eqref{mm2} with $V_{t_i}^\frac{p}{2}\theta_j$, we can derive
    \begin{align}\label{mm5}
        |d_{ij}|\lesssim&\;o_\delta(1)\sum_{(k,l)\neq (i,j)}|d_{kl}|+Q\norm*{h}_1\chi(p<4)+Q\norm*{h}_2\chi(p\geq 4)\nonumber\\
        &+o_\delta(1)Q\left(\norm*{\phi}_3\chi(p<4)+\norm*{\phi}_2\chi(p\geq 4)\right).
    \end{align}
    From the estimates \eqref{mm3}, \eqref{mm4} and \eqref{mm5} we know that \eqref{mm1} still holds for the equation \eqref{mm2}, and we have
    \begin{align}\label{mm6}
        \sum_{i=1}^\nu|d_i|+\sum_{i=1}^\nu\sum_{j=1}^n|d_{ij}|\lesssim Q\norm*{h}_1\chi(p<4)+Q\norm*{h}_2\chi(p\geq 4).
    \end{align}
    
    Thanks to the estimates \eqref{fin1}, \eqref{fin2}, \eqref{mm1} and \eqref{mm6}, when $\delta$ is small enough, we can apply a standard fix-point argument as in \cite[Proposition 3.8]{Den} to derive that there exist a function $\rho_0\in H^1(\C)$ and two sequences of numbers $\{c_i\}_i^\nu,\{c_{ij}\}^{\nu,n}_{i,j}$ such that
        \begin{align}\label{eq3}
            -\partial_s^2 \rho_0-\Delta_\theta \rho_0+\Lambda \rho_0-p\sigma^{p-2}\rho_0 =&\;\sigma^{p-1}-\sum_{i=1}^\nu V_{t_i}^{p-1}\nonumber\\
            &+|\sigma+\rho_0|^{p-2}(\sigma+\rho_0)-\sigma^{p-1}-p\sigma^{p-2}\rho_0\\
            &+\sum_{i=1}^\nu c_iV_{t_i}^{p-2}\partial_sV_{t_i}+\sum_{i=1}^\nu\sum_{j=1}^n c_{ij}V_{t_i}^{p-2}V_{t_i}^{\frac{p}{2}}\theta_j\nonumber
        \end{align}
        and
    \begin{align}\label{fin3}
        \left\langle \rho_0,\partial_s V_{t_i}\right\rangle_{H^1(\C)}=\left\langle \rho_0, V^{\frac{p}{2}}_{t_i}\theta_j\right\rangle_{H^1(\C)}=\left\langle \rho_0,V_{t_i}^{p-2}\partial_s V_{t_i}\right\rangle_{L^2(\C)}=\left\langle \rho_0,V_{t_i}^{p-2}V^{\frac{p}{2}}_{t_i}\theta_j\right\rangle_{L^2(\C)}=0
    \end{align}
    for any $1\leq i\leq \nu,1\leq j\leq n$. Moreover, we have
    \begin{equation}\label{fin4}
        \norm*{\rho_0}_3\chi(p<4)+\norm*{\rho_0}_2\chi(p\geq 4)\lesssim 1.
    \end{equation}
    Let us test \eqref{eq3} with $\partial_sV_{t_i}$, $V_{t_i}^{\frac{p}{2}}\theta_j$ and $\rho_0$. Exploiting Lemma \ref{le1}, Lemma \ref{le5}, the estimates \eqref{fin1}, \eqref{fin2},\eqref{fin3}, \eqref{fin4} and the fact $|\partial_s V_{t_i}|\lesssim V_{t_i}$ yields
        \begin{align}
            |c_i|\int_\C V_{t_i}^{p-2}(\partial_sV_{t_i})^2+o_\delta(1)\sum_{j\neq i}c_j\lesssim \chi(p< 4)\int_\C (W_1+W_3)V_{t_i}+\chi(p\geq 4)\int_\C W_2V_{t_i}
            \lesssim Q,\nonumber
        \end{align}
        \begin{align}
            |c_{ij}|\int_\C V_{t_i}^{p-2}(V_{t_i}^{\frac{p}{2}}\theta_j)^2+o_\delta(1)\sum_{k\neq i}c_{kj}\lesssim\chi(p< 4)\int_\C W_3V_{t_i}+\chi(p\geq 4)\int_\C W_2V_{t_i}
            \lesssim Q\nonumber
        \end{align}
    and 
    \begin{align}
         \norm*{\rho_0}_{H^1(\C)}^2 \lesssim&\;\int_\C \sigma^{p-2}\rho_0^2+\chi(p<4)\int_\C W_1W_3+\chi(p\geq 4)\int_C W_2^2+\norm*{\rho_0}_{H^1(\C)}^{\min\{3,p\}}\nonumber\\
            \lesssim&\;\sum_{i=1}^\nu\int_\C V_{t_i}^{p-2}\left(\chi(p<4)W^2_3+\chi(p\geq 4)W^2_2\right)+F_1^2(Q)+\norm*{\rho_0}_{H^1(\C)}^{\min\{3,p\}}\nonumber\\
            \lesssim&\;F_1^2(Q)+\norm*{\rho_0}_{H^1(\C)}^{\min\{3,p\}}.\nonumber
    \end{align}
    Therefore, when $\delta$ is sufficiently small, we have
    \begin{align}\label{fin5}
        \sum_{i=1}^\nu |c_i|+\sum_{i=1}^\nu\sum_{j=1}^n |c_{ij}|\lesssim Q,\quad \norm*{\rho_0}_{H^1(\C)}\lesssim F_1(Q).
    \end{align}

    Define $\rho_1:=\rho-\rho_0\in H^1(\C)$. It remains to estimate $\norm*{\rho_1}_{H^1(\C)}$. Note that $\rho_1$ satisfies the equation
    \begin{equation}\label{eq4}
        \begin{aligned}
            -\partial_s^2 \rho_1-\Delta_\theta \rho_1+\Lambda \rho_1=&\;|\sigma+\rho_0+\rho_1|^{p-2}(\sigma+\rho_0+\rho_1)-|\sigma+\rho_0|^{p-2}(\sigma+\rho_0)\\
            &-\sum_{i=1}^\nu c_iV_{t_i}^{p-2}\partial_sV_{t_i}-\sum_{i=1}^\nu\sum_{j=1}^n c_{ij}V_{t_i}^{p-2}V_{t_i}^{\frac{p}{2}}\theta_j-H_1(v).\\
        \end{aligned}
    \end{equation}
    Consider
    \begin{align}
        \rho_1=\sum_{i=1}^\nu \beta_i V_{t_i}+\sum_{i=1}^\nu \hat{\beta}_i \partial_sV_{t_i}+\sum_{i=1}^\nu\sum_{j=1}^n \beta_{ij}V_{t_i}^{\frac{p}{2}}\theta_j+\rho_2,\nonumber
    \end{align}
    where $\rho_2$ satisfies
    \begin{align}\label{bu1}
        \left\langle\rho_2, V_{t_i}\right\rangle_{H^1(\C)}= \left\langle \rho_2,\partial_s V_{t_i}\right\rangle_{H^1(\C)}=\left\langle \rho_2,V_{t_i}^{\frac{p}{2}}\theta_j\right\rangle_{H^1(\C)}=0,\quad\forall 1\leq i\leq \nu,1\leq j\leq n.
    \end{align}
    Since $\rho_1$ is orthogonal to $\partial_sV_{t_i}$ for any $1\leq i\leq \nu$, we have
    \begin{equation}\label{bu2}
        \sum_{i=1}^\nu |\hat{\beta}_i|\lesssim Q\sum_{i=1}^\nu |\beta_i|.
    \end{equation}
    We also have
    \begin{align}\label{bu3}
        \Pi_{Y}(v)=\sum_{i=1}^\nu\sum_{j=1}^n \beta_{ij}V_{t_i}^{\frac{p}{2}}\theta_j.
    \end{align}
    By testing \eqref{eq4} with $V_{t_i}$, using Lemma \ref{le1}, Lemma \ref{le5} and recalling the estimates \eqref{fin4}, \eqref{fin5}, \eqref{bu1}, \eqref{bu2}, we can deduce
        \begin{align}\label{fin7}
            &\;(p-2)|\beta_i|\norm*{V_{t_i}}_{H^1(\C)}^2+o_\delta(1)\sum_{j\neq i}\beta_j\nonumber\\
            \leq&\; (p-1)\left|\int_\C \left(|\sigma+\rho_0|^{p-2}-V_{t_i}^{p-2}\right)\rho_1V_{t_i}\right|+C\norm*{\rho_1}_{H^1(\C)}^{\min\{2,p-1\}}+CQ^2+C\norm*{H_1(v)}_{H^{-1}(\C)}\nonumber\\
            \lesssim&\;\chi(p\geq 3)\sum_{j\neq i}\int_\C \left(V_{t_j}^{p-2}V_{t_i}+V_{t_j}V_{t_i}^{p-2}\right)|\rho_1|+\chi(p\geq 3)\int_\C \left(|\rho_0|^{p-2}V_{t_i}+|\rho_0|V_{t_i}^{p-2}\right)|\rho_1|\nonumber\\
            &+\chi(p<3)\sum_{j\neq i}\int_\C V_{t_j}^{\frac{p-1}{2}}V_{t_i}^{\frac{p-1}{2}}|\rho_1|+\chi(p<3)\int_\C |\rho_0|^{\frac{p-1}{2}}V_{t_i}^{\frac{p-1}{2}}|\rho_1|\nonumber\\
            &+\sum_{j=1}^\nu |\beta_j|^{\min\{2,p-1\}}+\sum_{j=1}^\nu\sum_{k=1}^n |\beta_{jk}|^{\min\{2,p-1\}}+\norm*{\rho_2}_{H^1(\C)}^{\min\{2,p-1\}}+\norm*{H_1(v)}_{H^{-1}(\C)}+Q^2\nonumber\\
            \lesssim&\;o_{\delta}(1)\sum_{j=1}^\nu|\beta_j|+o_\delta(1)\norm*{\rho_2}_{H^1(\C)}+Q^2+\norm*{H_1(v)}_{H^{-1}(\C)}\\
            &+\sum_{j=1}^\nu\sum_{k=1}^n |\beta_{jk}|^{\min\{2,p-1\}}+F_3(Q)\sum_{j=1}^\nu\sum_{k=1}^n |\beta_{jk}|\nonumber,
        \end{align}
    where
    \begin{align}
        F_3(x):=\begin{cases}
            x,&\text{if }p>3,\\
            x^{\frac{p-1}{2}}(-\ln{x})^{\frac{p-1}{p}},&\text{if }p\leq 3.
        \end{cases}\nonumber
    \end{align}
    Note that \eqref{fin7} is equivalent to 
    \begin{align}\label{we1}
        |\beta_i|\lesssim o_\delta(1)\norm*{\rho_2}_{H^1(\C)}+Q^2+\norm*{H_1(v)}_{H^{-1}(\C)}+\sum_{j=1}^\nu\sum_{k=1}^n |\beta_{jk}|^{\min\{2,p-1\}}+F_3(Q)\sum_{j=1}^\nu\sum_{k=1}^n |\beta_{jk}|.
    \end{align}
    Similarly, by testing \eqref{eq4} with $\rho_2$, we can obtain
    \begin{align}\label{fin8}
        &\norm*{\rho_2}_{H^1(\C)}^2-(p-1)\int_\C \sigma^{p-2}\rho_2^2\nonumber\\
        \leq&\;(p-1)\left|\int_\C |\sigma+\rho_0|^{p-2}\rho_1\rho_2-\sigma^{p-2}\rho_2^2\right|+C\left(\norm*{\rho_1}_{H^1(\C)}^{\min\{2,p-1\}}+\norm*{H_1(v)}_{H^{-1}(\C)}\right)\norm*{\rho_2}_{H^1(\C)}\nonumber\\
        \lesssim&\;\left|\int_\C \left(|\sigma+\rho_0|^{p-2}-\sigma^{p-2}\right)\rho_2^2\right|+\left(\sum_{j=1}^\nu|\beta_j|+\sum_{j=1}^\nu\sum_{k=1}^n |\beta_{jk}|\right)\int_\C\left||\sigma+\rho_0|^{p-2}-\sigma^{p-2}\right|\cdot\sigma|\rho_2|\nonumber\\
        &+\left(\norm*{\rho_1}_{H^1(\C)}^{\min\{2,p-1\}}+\norm*{H_1(v)}_{H^{-1}(\C)}\right)\norm*{\rho_2}_{H^1(\C)}\nonumber\\
        \lesssim&\;o_\delta(1)\norm*{\rho_2}_{H^1(\C)}^2+F_3(Q)\left(\sum_{j=1}^\nu|\beta_j|+\sum_{j=1}^\nu\sum_{k=1}^n |\beta_{jk}|\right) \norm*{\rho_2}_{H^1(\C)}\\
        &+\left(\norm*{\rho_1}_{H^1(\C)}^{\min\{2,p-1\}}+\norm*{H_1(v)}_{H^{-1}(\C)}\right)\norm*{\rho_2}_{H^1(\C)}\nonumber
    \end{align}
    It follows from Lemma \ref{le2} and a standard localization argument (see for example \cite[Proposition 3.10]{Fig}) that
    \begin{equation}\label{fin9}
        \norm*{\rho_2}_{H^1(\C)}^2\geq (\gamma_3-o_{\delta}(1))\int_\C \sigma^{p-2}\rho_2^2.
    \end{equation}
    Combining \eqref{fin8} and \eqref{fin9} gives
    \begin{align}\label{we2}
        \norm*{\rho_2}_{H^1(\C)}\lesssim o_{\delta}(1)\sum_{j=1}^\nu|\beta_j|+\norm*{H_1(v)}_{H^{-1}(\C)}+\sum_{j=1}^\nu\sum_{k=1}^n |\beta_{jk}|^{\min\{2,p-1\}}+F_3(Q)\sum_{j=1}^\nu\sum_{k=1}^n |\beta_{jk}|.
    \end{align}
    From \eqref{we1} and \eqref{we2} we get
    \begin{align}\label{we3}
        \sum_{i=1}^\nu|\beta_i|+\norm*{\rho_2}_{H^1(\C)}\lesssim Q^2+\norm*{H_1(v)}_{H^{-1}(\C)}+\sum_{j=1}^\nu\sum_{k=1}^n |\beta_{jk}|^{\min\{2,p-1\}}+F_3(Q)\sum_{j=1}^\nu\sum_{k=1}^n |\beta_{jk}|
    \end{align}
    Substituting \eqref{fin5}, \eqref{bu2} and \eqref{we3} into \eqref{lower bd} yields
    \begin{align}\label{we4}
        \norm*{H_1(v)}_{H^{-1}(\C)}+\sum_{j=1}^\nu\sum_{k=1}^n |\beta_{jk}|^2+\max_{1\leq k\leq\nu}\left|\int_\C\sigma^{p-2}\rho\;\partial_s V_{t_k}\right|\gtrsim Q.
    \end{align}
    Note that
    \begin{align}\label{we5}
        \left|\int_\C\sigma^{p-2}\rho\;\partial_s V_{t_k}\right|\lesssim&\; \left|\int_\C\sigma^{p-2}\rho_0\;\partial_s V_{t_k}\right|+\left|\int_\C\sigma^{p-2}\rho_1\;\partial_s V_{t_k}\right|\nonumber\\
        \lesssim&\;\left|\int_\C\left(\sigma^{p-2}-V_{t_k}^{p-2}\right)\rho_0\;\partial_s V_{t_k}\right|+\sum_{i=1}^\nu|\beta_i|\left|\int_\C\sigma^{p-2}V_{t_i}\;\partial_s V_{t_k}\right|\nonumber\\
        &+\sum_{i=1}^\nu|\hat{\beta}_i|\left|\int_\C\sigma^{p-2}\partial_sV_{t_i}\;\partial_s V_{t_k}\right|+  \left|\int_\C\sigma^{p-2}\rho_2\;\partial_s V_{t_k}\right|\nonumber\\
        \lesssim&\;o_\delta(1)Q+|\beta_k|\left|\int_\C\left(\sigma^{p-2}-V_{t_k}^{p-2}\right)V_{t_k}\;\partial_s V_{t_k}\right|\nonumber\\
        &+\left|\int_\C\left(\sigma^{p-2}-V_{t_k}^{p-2}\right)\rho_2\;\partial_s V_{t_k}\right|\nonumber\\
        \lesssim&\;o_\delta(1)Q+F_3(Q)\norm*{\rho_2}_{H^1(\C)}.
    \end{align}
    From \eqref{we3}, \eqref{we4} and \eqref{we5} we derive
    \begin{align}\label{we6}
        \norm*{H_1(v)}_{H^{-1}(\C)}+\sum_{j=1}^\nu\sum_{k=1}^n |\beta_{jk}|^2+F_3(Q)\sum_{j=1}^\nu\sum_{k=1}^n |\beta_{jk}|^{\min\{2,p-1\}}\gtrsim Q.
    \end{align}
    
    Now we define
    \begin{align}
        \psi_1:=\Pi_{Y}(v)=\sum_{i=1}^\nu\sum_{j=1}^n\beta_{ij}V_{t_i}^{\frac{p}{2}}\theta_j\nonumber
    \end{align}
    and
    \begin{align}
        \psi_2:=\Pi_{Y}^\perp(v)-\sigma=\rho_0+\sum_{i=1}^\nu \beta_i V_{t_i}+\sum_{i=1}^\nu \hat{\beta}_i \partial_sV_{t_i}+\rho_2.\nonumber
    \end{align}
    Based on the above arguments, we have
    the following two estimates:
    \begin{align}\label{hhh1}
        \norm*{\psi_2}_{H^1(\C)}\lesssim F_1(Q)+\norm*{H_1(v)}_{H^{-1}(\C)}+F_3(Q)\norm*{\psi_1}_{H^1(\C)}+\norm*{\psi_1}^{\min\{2,p-1\}}_{H^1(\C)}
    \end{align}
    and
    \begin{align}\label{hhh2}
        Q\lesssim\norm*{H_1(v)}_{H^{-1}(\C)}+F_3(Q)\norm*{\psi_1}^{\min\{2,p-1\}}_{H^1(\C)}+\norm*{\psi_1}^2_{H^1(\C)}.
    \end{align}
    If
    \begin{align}
        F_3(Q)\norm*{\psi_1}^{\min\{2,p-1\}}_{H^1(\C)}+\norm*{\psi_1}^2_{H^1(\C)}\leq \norm*{H_1(v)}_{H^{-1}(\C)},\nonumber
    \end{align}
    from \eqref{hhh2} we get
    \begin{align}\label{hhh3}
        Q\lesssim\norm*{H_1(v)}_{H^{-1}(\C)},\quad \norm*{\psi_1}_{H^{1}(\C)}\leq \norm*{H_1(v)}_{H^{-1}(\C)}^{\frac{1}{2}}.
    \end{align}
    Combining \eqref{hhh1} and \eqref{hhh3} yields
    \begin{align}
        \norm*{\psi_2}_{H^{1}(\C)}\lesssim F_1\left(\norm*{H_1(v)}_{H^{-1}(\C)}\right),\nonumber
    \end{align}
    which indicates \eqref{main-re6}. If
    \begin{align}
        F_3(Q)\norm*{\psi_1}^{\min\{2,p-1\}}_{H^1(\C)}+\norm*{\psi_1}^2_{H^1(\C)}\geq \norm*{H_1(v)}_{H^{-1}(\C)},\nonumber
    \end{align}
    from \eqref{hhh2} we obtain
    \begin{align}\label{hhh4}
        Q\lesssim \norm*{\psi_1}_{H^{1}(\C)}^2,\quad \norm*{H_1(v)}_{H^{-1}(\C)}\lesssim \norm*{\psi_1}_{H^{1}(\C)}^2.
    \end{align}
    Combining \eqref{hhh1} and \eqref{hhh4} gives
    \begin{align}
        \norm*{\psi_2}_{H^{1}(\C)}\lesssim F_1\left(\norm*{\psi_1}^2_{H^{1}(\C)}\right),\nonumber
    \end{align}
    which contradicts our assumptions. The proof is complete.
\end{proof}
\begin{proof}[Proof of Theorem \ref{thm2}]
    Thanks to the global compactness principle, there exists a number $\varepsilon=\varepsilon(a,n,\nu)$ such that, when
    \begin{align}\label{wer}
        \norm*{H(u)}_{D_a^{-1,2}(\R^n)}\leq \varepsilon,
    \end{align}
    we have
    \begin{align}
        \inf_{\substack{s_i\in\R_+\\
    1\leq i\leq\nu}}\norm*{u-\sum_{i=1}^\nu U_{s_i}}_{D_a^{1,2}(\R^n)}\leq \delta\nonumber
    \end{align}
    and the infimum can be attained by $U_{\lambda_i},1\leq i\leq \nu$ with
    \begin{align}
        \min\left\{\frac{\lambda_i}{\lambda_j},\frac{\lambda_j}{\lambda_i}\right\}\leq \delta,\quad \forall\;i\neq j.\nonumber
    \end{align}
    Thus when \eqref{wer} holds, \eqref{main-re3} comes from \eqref{main-re6}. If \eqref{wer} does not hold, \eqref{main-re3} follows from the fact that the left side of it has an obvious upper bound.
\end{proof}
\vspace{10pt}
\noindent{\Large\textbf{Declarations}}
~\\
~\\
\textbf{Conflict of interest}\quad On behalf of all authors, the corresponding author states that there is no Conflict of interest.~\\
{\textbf{Data Availability Statements}}\quad All data generated or analyzed during this study are included in this article.

\end{document}